\documentclass[a4paper,11pt]{article}

\usepackage[english]{babel}
\usepackage[utf8x]{inputenc}
\usepackage[T1]{fontenc}
\usepackage[doublespacing]{setspace}
\usepackage[left, mathlines]{lineno}

\usepackage{placeins}
\usepackage{float}
\restylefloat{table}
\restylefloat{figure}

\usepackage[a4paper,top=2cm,bottom=2cm,left=2cm,right=2cm,marginparwidth=1.75cm]{geometry}

\usepackage{lmodern}
\usepackage{subcaption}
\usepackage{mathtools, amsthm, amssymb}
\usepackage[dvipsnames]{xcolor}
\usepackage{graphicx}
\usepackage{array}
\usepackage{calc}
\usepackage{dsfont}
\usepackage[nottoc, notlof, notlot]{tocbibind}
\usepackage{stmaryrd}

\usepackage{amsmath}
\usepackage{amsthm}
\usepackage{graphicx}
\usepackage{hyperref}
\hypersetup{colorlinks,allcolors=black}
\usepackage{float}

\usepackage[numbers]{natbib}

\usepackage{authblk} %Pour les auteurs

\allowdisplaybreaks[4] %pour que les equations occupent plusieurs pages
\binoppenalty=\maxdimen
\relpenalty=\maxdimen

\theoremstyle{plain}
\newtheorem{thm}{Theorem}
\newtheorem{lem}[thm]{Lemma}
\newtheorem{prop}[thm]{Proposition}

\newtheorem{cor}[thm]{Corollary}
\newtheorem{defn}[thm]{Definition}

\theoremstyle{remark}
\newtheorem{rem}[thm]{Remark}

\newtheorem*{sketch}{\textbf{Sketch of the proof}}

\numberwithin{thm}{section}
\numberwithin{equation}{section}

\newenvironment{AMS}{\textbf{\textit{MSC 2020 Subject Classification:}}}{}
\newenvironment{keywords}{\textbf{\textit{Keywords:}}}{}
\newenvironment{acknowledgements}{\textbf{Acknowledgements}}{}
\newenvironment{headline}{\textbf{\textit{Running headline:}}}{}

\newcommand{\proba}{\mathbb{P}}
\newcommand{\nmath}{\mathbb{N}}
\newcommand{\zmath}{\mathbb{Z}}

\newcommand{\fcal}{\mathcal{F}}
\newcommand{\hcal}{\mathcal{H}}
\newcommand{\lcal}{\mathcal{L}}

\newcommand{\vcal}{\mathcal{V}}

\newcommand{\Ext}{\mathrm{Ext}}
\newcommand{\Card}{\mathrm{Card}}
\newcommand{\dboa}{\mathrm{DiffBOA}}
\newcommand{\errpar}{\mathrm{Par}}
\newcommand{\errf}{\mathrm{ErrF}}
\newcommand{\errg}{\mathrm{ErrG}}
\newcommand{\errgf}{\mathrm{ErrGF}}
\newcommand{\cond}{\mathrm{Cond}}

\newcommand{\un}[1]{\mathds{1}_{\{#1\}}}

\newcommand{\lgmr}{\lfloor gM \rfloor}

\title{Extinction threshold and large population limit of a plant metapopulation model with recurrent extinction events and a seed bank component}
\date{}
\author[,1]{Apolline Louvet \thanks{Email address : apolline.louvet@polytechnique.edu}}
 
\affil[1]{MAP5, Universit{\'e} de Paris, CNRS, 45 rue des Saints-P{\`e}res, 75270 Paris Cedex 06, France.}

\begin{document}
\maketitle

\begin{abstract}
We introduce a new model for plant metapopulations with a seed bank component, living in a fragmented environment in which local extinction events are frequent. This model is an intermediate between population dynamics models with a seed bank component, based on the classical Wright-Fisher model, and Stochastic Patch Occupancy Models (SPOMs) used in metapopulation ecology. 
Its main feature is the use of "ghost" individuals, which can reproduce but with a very strong selective disadvantage against "real" individuals, to artificially ensure a constant population size. We show the existence of an extinction threshold above which persistence of the subpopulation of "real" individuals is not possible, and investigate how the seed bank characteristics affect this extinction threshold. We also show the convergence of the model to a SPOM under an appropriate scaling, bridging the gap between individual-based models and occupancy models. 
\end{abstract}

\begin{headline}
The $k$-parent WFSB metapopulation model
\end{headline}



\tableofcontents

\clearpage

\section{Introduction}
Understanding how plant populations survive in fragmented landscapes is an important question in ecology and conservation biology \cite{fahrig2003effects}. One potential driver of plant populations' persistence is the ability to form a seed bank, which greatly influences population and community dynamics \cite{fenner1995ecology}. For such plant species, the seeds produced can stay dormant in the soil for up to several decades depending on the species, without losing viability \cite{baskin2014seeds}. See \cite{lennon2020principles} for an overview of seed bank characteristics and properties, along with the emergent phenomena it can generate.

Populations living in fragmented landscapes are often modelled as metapopulations, that is, as populations distributed over a set of interconnected patches. Metapopulations are also frequently characterized by recurrent local extinction events, regional persistence being the result of a balance between colonization (from neighbouring patches or from an external source) and local extinction events \cite{levins1969some,macarthur1967theory}. See \cite{hanski1997metapopulation} for a general introduction to metapopulation theory.

Many classical metapopulation models, such as the Levins model \cite{levins1969some} or the Propagule Rain model \cite{gotelli1991metapopulation}, describe the occupancy of each patch (i.e whether the species of interest is present or absent in each of the patches) and do not depend on, nor model, the actual census numbers. 
These models are referred to as Stochastic Patch Occupancy models, or SPOMs. Since presence/absence data is easier to collect than abundance data, and since parameter inference is possible for a broad range of SPOMs (see e.g \cite{moilanen1999patch,moilanen2004spomsim,pluntz2018general}), they are well-suited to the study of real metapopulations. Classical metapopulation models do not account for seed dormancy, but more recently models incorporating a seed bank component were also developed  \cite{borgy2015dynamics,freville2013inferring,pluntz2018general}. The model introduced in \cite{pluntz2018general} was successfully applied to plant metapopulations in highly disturbed environments, such as weeds in agroecosystems \cite{pluntz2018general} or plants in urban tree bases \cite{louvet2021detecting}, highlighting that some plant species monitored did have a seed bank.

In population genetics, metapopulation models often describe the number and genetic types of individuals rather than the occupancy in each patch.
They are usually defined by first specifying an intra-patch dynamic, and then adding migration between patches. The migration process can  heavily depend on the underlying geographical structure, as in the stepping-stone model \cite{kimura1964stepping}, or not depend on it at all, as in Wright's island model \cite{wright1931evolution}. See e.g \cite{lambert2015coalescent,taylor2009coalescent,wakeley2001gene} and references therein for examples of metapopulation models based on Wright's island model, and \cite{austerlitz1997evolution,barton2013genetic,peischl2015expansion} and references therein for examples of metapopulation models based on the stepping-stone model. 

Models used to specify the intra-patch dynamic can be classical population dynamics models, without any intra-patch spatial structure, provided patches are considered as sufficiently small to neglect spatial effects in each one of them. The geographical structure in the metapopulation model is then only contained in the localization of the patches. 
The intra-patch dynamic can comprise a seed bank, using population dynamics models with a seed bank component, such as the ones based on the Wright-Fisher model.
In the original Wright-Fisher model, the population size (in a single patch) is constant through time and equal to $N$, and each individual has a genetic type, or allele. In each generation, each one of the $N$ new individuals chooses a parent uniformly at random among the $N$ individuals in the previous generation, and adopts its type. Including a seed bank in the Wright-Fisher model implies choosing a parent potentially not in the previous generation, but at least two generations ago, the maximal number of potentially contributing generations being bounded \cite{kaj2001coalescent} or not \cite{blath2013ancestral,blath2016new}. 
%It is also possible to introduce selection in favour of one of the types, i.e to give one type a better capacity to produce offspring compared to the other types \cite{koopmann2017fisher}. 
See \cite{blath2020population} for a review of seed bank models in population genetics, and \cite{den2017multi,greven2020spatial,vzivkovic2012germ} for extensions of the Wright-Fisher model with a seed bank component to metapopulations. 

For plant metapopulations in which extinction events are frequent, we can expect the population size of each patch to vary a lot from one generation to the next. This contradicts the constant patch population size hypothesis underlying the use of a Wright-Fisher model. 
In order to incorporate extinction event-induced fluctuations in a Wright-Fisher model, it is possible to adopt the approach used in \cite{durrett2016genealogies,hallatschek2008gene}, based on a long-range biased voter model: assign a maximal population size to each patch, and fill the remaining space with "ghost", or type $0$, individuals. In this framework, each patch contains both type $1$ "real" individuals and type $0$ "ghost" individuals, the former having a very strong selective advantage over the latter (in the spirit of \cite{louvet2021k}). As in \cite{durrett2016genealogies,hallatschek2008gene}, the model can then be interpreted as an interacting particle system, in which a $0$ corresponds to a ghost individual and a $1$ to a real individual.

In this article, we introduce a new individual-based metapopulation model for plant metapopulations in which local extinction events are frequent.
This model is primarily suited to annual plants living in highly disturbed patchy environments, such as urban tree bases or agroecosystems. It is also adapted to other plant species living in such environments, provided each patch is "emptied" at the end of each generation (for instance by gardeners in an urban environment or by farmers in an agroecosystem).
The intra-patch dynamics will be based on a variant of a Wright-Fisher model with a seed bank component, using ghost individuals to allow for fluctuating patch population sizes. It will use the model introduced in \cite{blath2016new}, with an extra bound introduced on the number of generations a seed can stay dormant without losing viability. Indeed, for some plant species, seeds lose viability after only one or two years of dormancy \cite{baskin2014seeds}. This is reminiscent of the model introduced in \cite{kaj2001coalescent}, in the sense that real individuals come from parents living a bounded number of generations ago. The main difference is that in our model, individuals first choose parents living potentially arbitrarily far ago in the past, and then obtain their (real or ghost) type depending on their choice. 
%Although this is reminiscent of the model introduced in \cite{kaj2001coalescent}, the main difference is that in our model, even though "real" individuals do come from parents living a bounded number of generations ago, individuals of unknown types may come from a parent living arbitrarily far ago in the past.

In order to bridge the gap between individual-based metapopulation models and SPOMs, we will show that our metapopulation process can be embedded in a SPOM. Moreover, we will prove that under an appropriate scaling of patch population size and of the "selection" strength of real individuals against ghost individuals (quantified by the parameter $k$), the individual-based metapopulation process converges to this SPOM. The convergence result has two applications. First, from a theoretical viewpoint, it shows that a specific SPOM (or presence/absence-based model) is the scaling limit of an individual-based metapopulation model. Then, we will use the convergence result and the embedding in order to show the existence of an extinction threshold for metapopulation persistence, depending only on the seed bank parameters, and highlighting how the presence of a seed bank can prevent metapopulation extinction.

While the metapopulation model we will introduce and study is based on models coming from population genetics, this article will not focus on the study of the genetic diversity in such populations, which is deferred to future work. Instead, the aims of this work are threefold:
\begin{enumerate}
\item Introduce a general individual-based metapopulation model with a seed bank component, in which local extinction events can be frequent and patch population sizes can vary from one generation to the next.
\item Show the existence of an extinction threshold depending on the seed bank parameters.
\item Bridge the gap between SPOMs and individual-based metapopulation models by showing that in a well-chosen parameter regime, the individual-based metapopulation model we consider converges to a SPOM.
\end{enumerate}

\subsection{The $k$-parent Wright-Fisher metapopulation process with seed bank}
%All the objects introduced in this section will be defined on the same probability space $(\Omega, \fcal, \proba)$. 
In the whole paper, we use the notation $\nmath = \{0, 1, 2,...\}$, and for all $M \in \nmath \backslash \{0\}$, 
$\llbracket 1,M \rrbracket = \{1,...,M\}$. 

We will consider that the metapopulation is formed by an infinite number of patches arranged in a line. A patch contains a fixed number of \textit{seed bank compartments}, each one containing exactly one seed: either a ghost (type $0$) seed, or a real (type $1$) seed. In order to define the metapopulation model, we describe how in each generation, seeds germinate and grow into plants which produce new seeds and die. Concretely, the metapopulation model will only record the composition of the \textit{seed bank} at the beginning of each generation, and not the standing vegetation in each patch in each generation. 

In all that follows, let $M \in \nmath^{*}$, $H \in \nmath$, $k \in \nmath \backslash \{0,1\}$, $g \in (0,1)$, $c \in (0,1/2)$ and $p \in [0,1]$. We assume that $\lfloor gM \rfloor \geq 1$. Patches will be indexed by $i \in \zmath$, and seed bank compartments inside a patch by $j \in \llbracket 1,M \rrbracket$. The notation $(i,j)$ will correspond to the seed bank compartment $j$ in patch $i$.

The following two spaces will be used to describe the initial types and the age of the seeds occupying the seed bank compartments:
\begin{align*}
\fcal_{M} &:= \left\{(\xi_{i,j})_{i \in \zmath, j \in \llbracket 1,M \rrbracket} \in \{0,1\}^{\zmath \times \llbracket 1,M \rrbracket} :
\mathrm{Card} \left(\left\{
(i,j) \in \zmath \times \llbracket 1,M \rrbracket : \xi_{i,j} = 1
\right\}
\right) < + \infty
\right\}, \\
\text{and } \hcal_{M} &:= \left\{
(h_{i,j})_{i \in \zmath, j \in \llbracket 1,M \rrbracket} : \forall (i,j) \in \zmath \times \llbracket 1,M \rrbracket, h_{i,j} \in \nmath
\right\}.
\end{align*}
%Here $\nmath = \{0, 1, 2,... \}$ and $\llbracket 1,M \rrbracket = \{1,..., M\}$.

$(\xi, h) \in \fcal_{M} \times \hcal_{M}$ corresponds to a metapopulation in which for all $(i,j) \in \zmath \times \llbracket 1,M \rrbracket$, the seed occupying the seed bank compartment $(i,j)$ is of age $h_{i,j}$ and of type $\xi_{i,j}\un{h_{i,j} \leq H}$. That is, the seed in $(i,j)$ was originally of type $\xi_{i,j}$ when it was produced, but may have expired since then. 

The $k$-parent Wright-Fisher metapopulation process with seed bank is defined in the following way. 

\begin{defn}($k$-parent WFSB metapopulation process) 
Let $(\xi,h) \in \fcal_{M} \times \hcal_{M}$. The $k$-parent Wright-Fisher metapopulation process with seed bank, with parameters $(M,H,g,c,p)$ and initial condition $(\xi,h)$ and denoted by $(\xi^{n},h^{n})_{n \in \nmath}$, is the $\left(\fcal_{M} \times \hcal_{M}\right)$-valued Markov chain defined by $(\xi^{0},h^{0}) = (\xi,h)$ and for all $n \in \nmath$, given $(\xi^{n},h^{n})$ :
\begin{enumerate}
\item For each $i \in \zmath$, we sample $\lfloor gM \rfloor$ different seed bank compartments $s_{i,1},...,s_{i,\lfloor gM \rfloor} \in \llbracket 1,M \rrbracket$ uniformly at random in patch $i$. 
\item Let $(\mathrm{Ext}_{i})_{i \in \zmath}$ be i.i.d $\{0,1\}$-valued random variables such that $\proba(\Ext_{1} = 1) = p$. 
\item For all $(i,j) \in \zmath \times \llbracket 1,M \rrbracket$, if $j \in \{s_{i,j'} : j' \in \llbracket 1,\lfloor gM \rfloor \rrbracket\}$, we first set $h_{i,j}^{n+1} = 0$.
Moreover, let $C_{1,i,j},...,C_{k,i,j}$ be i.i.d $\{-1,0,1\}$-valued random variables such that
\begin{equation*}
\proba(C_{1,i,j} = 1) = \proba(C_{1,i,j} = -1) = c,
\end{equation*}
and such that $(C_{l,i,j})_{1 \leq l \leq k}$ is independent from $(\mathrm{Ext}_{i})_{i \in \zmath}$ and from $(C_{i',j',l})_{1 \leq l \leq k}$ for $(i',j') \neq (i,j)$. 

For all $l \in \llbracket 1,k \rrbracket$, if $\Ext_{i+C_{l,i,j}} = 1$, we set $\tilde{k}_{l} = 0$, and if $\Ext_{i+C_{l,i,j}} = 0$, we sample one seed bank compartment $j_{l}$ uniformly at random among the $\lgmr$ ones sampled in the patch $i + C_{l,i,j}$ (those in the set $\{s_{i+C_{l,i,j},j'} : j' \in \llbracket 1,\lfloor gM \rfloor \rrbracket\}$), and we set
\begin{equation*}
\tilde{k}_{l} = \xi_{i+C_{l,i,j},j_{l}}^{n}  
\mathds{1}_{\{h_{i+C_{l,i,j},j_{l}}^{n} \leq H\}}.
\end{equation*}
We conclude by setting $\xi_{i,j}^{n+1} = \max\{\tilde{k}_{l} : l \in \llbracket 1,k \rrbracket \}$.
\item On the other hand, if $j \notin \{s_{i,j'} : j' \in \llbracket 1,\lfloor gM \rfloor \rrbracket\}$, we set $\xi_{i,j}^{n+1} = \xi_{i,j}^{n}$ and $h_{i,j}^{n+1} = h_{i,j}^{n} + 1$.
\end{enumerate}
\end{defn}

Intuitively, the $k$-parent WFSB metapopulation process evolves as follows.
\begin{enumerate}
\item At each generation, exactly $\lfloor gM \rfloor$ seeds germinate in each patch. Type $0$ seeds yield (ghost) type $0$ plants, while type $1$ seeds yield (real) type $1$ plants \textit{only if the seed was produced less than $H+1$ generations ago}, i.e, only if it has not expired.
\item Then, each patch is affected by an extinction event independently from other patches and with probability $p$. During an extinction event, all the juvenile plants in the patch become type $0$ plants. 
\item In each patch, the $\lfloor gM \rfloor$ empty seed bank compartments are filled with new seeds in the following way. For each compartment, $k$ potential parents are chosen uniformly at random, each one of them being chosen in the same patch with probability $1-2c$, or in the patch on the left (resp. on the right) with probability $c$. The same potential parent may be chosen more than once for the same seed bank compartment. If all the $k$ plants chosen as potential parents are of type $0$, then the seed bank compartment is filled with a type $0$ seed produced by the last plant chosen. Conversely, if at least one of the $k$ plants chosen is of type $1$, then the first type $1$ plant chosen produces a seed which fills the seed bank compartment.
\item The remaining seeds stay dormant in the seed bank until the next generation. 
\end{enumerate}
See Figure \ref{fig:k_parent_WFSB_schema} for an illustration of this dynamics. 
As mentioned above, observe that while the dynamics involves seeds germinating, growing into plants which produce new seeds and then die, the model only encodes the \textit{seed bank composition}, and not the types of the plants. 

\begin{figure}[ht]
\centering
\includegraphics[width = \linewidth]{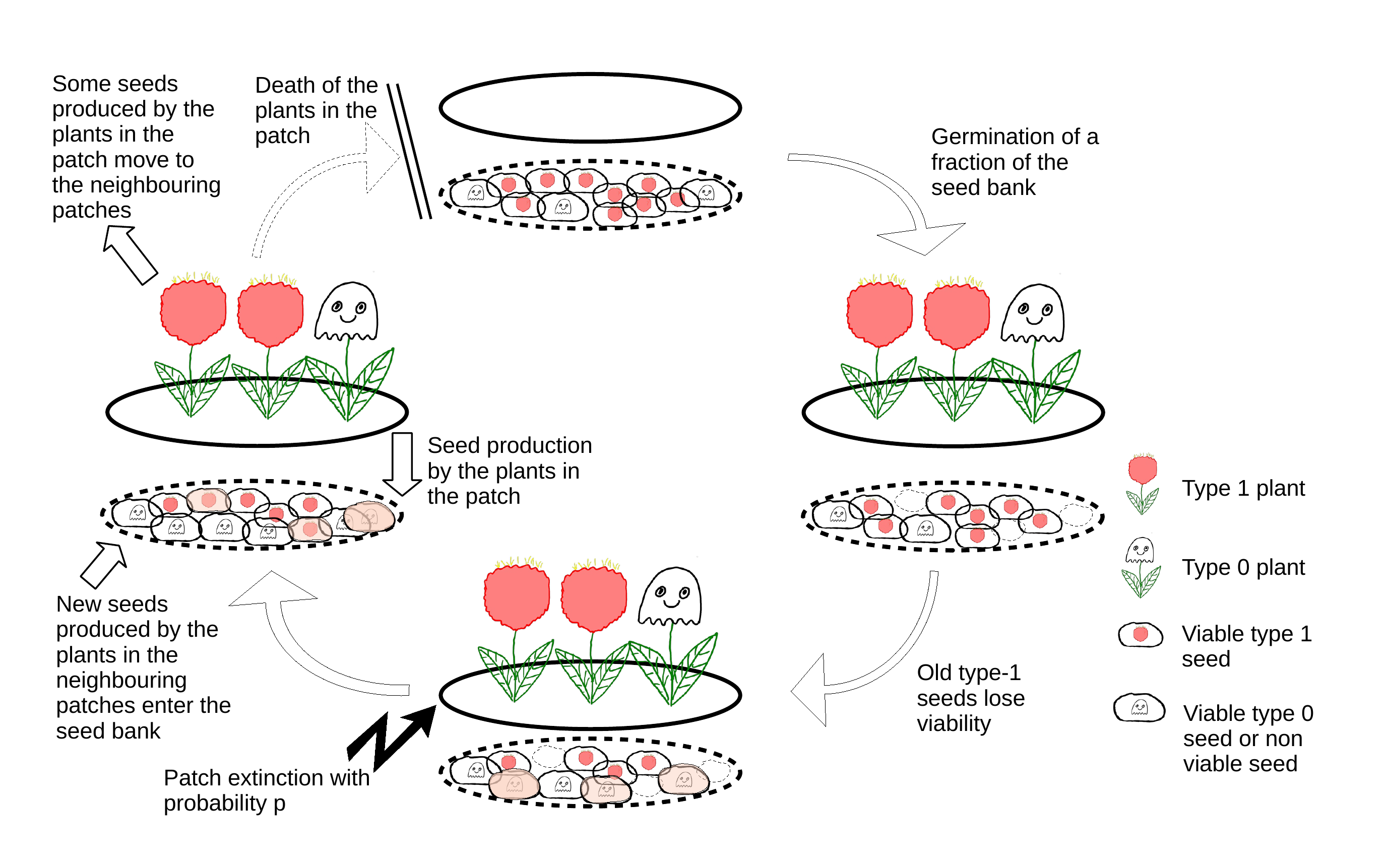}
\caption{Illustration of the intra-patch dynamics of the $k$-parent WFSB metapopulation process. Here $M = 12$ and $\lfloor gM \rfloor = 3$. The double line in the top line of the figure indicates the starting time of a new generation.}\label{fig:k_parent_WFSB_schema}
\end{figure}

In all that follows, we will refer to :
\begin{itemize}
\item $M$ as the \textit{number of seeds per patch},
\item $H$ as the \textit{maximal dormancy duration},
\item $g$ as the \textit{germination probability},
\item $c$ as the \textit{potential colonization probability},
\item $p$ as the \textit{patch extinction probability.}
\end{itemize}
Moreover, we will say that a patch \textit{goes extinct} during generation $n$ if it is affected by an extinction event during this generation (i.e $\textrm{Ext}_{i} = 1$), and that it is \textit{empty} if it does not contain any viable type $1$ seeds at the beginning of generation $n$. Notice that extinction events only affect standing vegetation, and not the seed bank. Therefore, in our terminology, an extinct patch is not necessarily empty. 
We will also say that the metapopulation \textit{became empty} before generation $n$ if all the patches are empty in generation~$n$, that is, if the subpopulation of real individuals did not persist. 

\begin{rem}
From a purely mathematical viewpoint, it is also possible to define the model for $k = 1$. However, in this case, real individuals do not have any selective advantage against ghost individuals, and since we assume that the number of real individuals in generation~$0$ is finite, the metapopulation becomes empty in finite time almost surely. 
\end{rem}

\begin{rem}
Even if the model is defined for $c \in (0,1/2)$, in practice, since one of the assumptions behind the model is that colonization from patches which are not nearest neighbours is negligible, it is implicitly assumed that $c$ is small. Moreover, notice that if $c > 1/3$, then the potential parents have higher chance of being taken from the patch on the left (or right) than in the focal patch.
\end{rem}

Before stating our main results on the model, we give some insight on the reproduction dynamics from a \textit{forwards in time} viewpoint. In order to do so, we first consider a real plant isolated in a patch, surrounded by empty patches. Each empty seed bank compartment in the focal patch chooses the real plant as a potential parent with probability $1-(1-(1-2c) \lfloor gM \rfloor^{-1})^{k}$, and each empty seed bank compartment in one of the two neighbouring patches chooses it with probability $1-(1-c \lfloor gM \rfloor^{-1})^{k}$. Therefore, from a forwards in time viewpoint, the number of offspring of the real plant is distributed as the sum of three independent random variables: one binomial random variable with parameters $(\lfloor gM \rfloor, 1-(1-(1-2c)\lfloor gM \rfloor^{-1})^{k})$, and two binomial random variables with parameters $(\lfloor gM \rfloor, 1-(1-c\lfloor gM \rfloor^{-1})^{k})$.

If $M \to + \infty$ while all other parameters stay constant, the number of offspring is approximately distributed as the sum of three independent Poisson random variables: one with parameter $(1-2c)k$ and two with parameter $ck$. In particular, the average number of offspring is roughly equal to $k$, yielding a possible biological interpretation for the parameter $k$. In this article, we focus instead on another parameter regime: $M \to + \infty$ and $k = \lceil M^{\alpha} \rceil$, $\alpha > 1$. In this regime, the average number of offspring is rather equal to $3 \lfloor gM \rfloor$. From a biological viewpoint, this parameter regime corresponds to considering that a plant can potentially produce far more viable seeds than the carrying capacity of a patch. 

When several real plants are present, the average number of offspring produced by each plant decreases due to competition. In the case in which competition is the most intense (that is, when the patch and its two neighbours contain only real plants), the number of offspring of each plant in the focal patch is in average equal to $1$, and is distributed as the sum of three independent random variables: one binomial random variable with parameters $(\lfloor gM \rfloor, (1-2c)\lfloor gM \rfloor^{-1})$, and two binomial random variables with parameters $(\lfloor gM \rfloor, c\lfloor gM \rfloor^{-1})$. When $M \to + \infty$ and $k$ is fixed, this can be approximated by the sum of a Poisson random variable with parameter $1-2c$ and two Poisson random variables with parameter $c$. In particular, the random variables do not depend on the parameter $k$. 

\begin{rem}
It is possible to generalize the $k$-parent WFSB metapopulation process by taking the potential parents of a seed in more patches than only neighbouring patches, or by having patches in a two dimensional environment instead of a one dimensional one. If the distance that seeds can travel is bounded, then all the results in this article can be extended to the generalized model (though the numerical values for the extinction thresholds will change). 
\end{rem}

\begin{rem}
The idea of sampling several potential parents to model selection can be found in various population genetics models, including variants of the Wright-Fisher model. See e.g 
\cite{boenkost2020haldane,boenkost2021haldane,cordero2019general,forien2017central,casanova2020lambda,casanova2018duality}.
Usually, the models comprise both selective reproduction events, during which several potential parents are chosen, and neutral reproduction events, during which only one parent is chosen. Moreover, the mathematical analysis often involves taking selective reproduction events to be rare compared to neutral reproduction events, and to change of time scale to observe them in the limit. 
In contrast, the model we introduce in this article only comprises selective reproduction events, and the questions we aim at answering do not require a change of time scale.
\end{rem}

\subsection{The associated $k$-parent occupancy process and its limit}
\subsubsection{BOA process and $k$-parent occupancy process}
The $k$-parent WFSB metapopulation process can be seen as a multi-colony Wright-Fisher model with selection and seed bank, embedded in a Stochastic Patch Occupancy Model indicating which patches are empty, extinct, or potentially occupied. 
Indeed, for a patch to contain real seeds, it is not sufficient for it not to go extinct. The viable seeds it contains can only come from $3$ patches (the focal patch and its two neighbours), and can only have entered the seed bank during the $H+1$ previous generations. If all these times, the $3$ patches were affected by extinction events, then the patch cannot contain viable seeds during the current generation. For instance, if $H = 0$, a patch which went extinct along with its two neighbours during the previous generation cannot contain non-expired type $1$ seeds. In the SPOM we define just below, this patch will appear as empty. In other words, the SPOM will encode which patches \textit{cannot} contain type $1$ seeds, given the initial condition and the extinction events.

This SPOM is defined on the state space $\fcal^{\infty} \times \hcal^{\infty}$, with $\fcal^{\infty}$ and $\hcal^{\infty}$ given by:
\begin{align*}
\fcal^{\infty} &:= \left\{
(O_{i})_{i \in \zmath} : \forall i \in \zmath, O_{i} \in \{0,1\} \text{ and } \mathrm{Card}\left(\left\{
i \in \zmath : O_{i} = 1
\right\}\right) < + \infty \right\} \\
\text{and } \hcal^{\infty} &:= \left\{
(h_{i})_{i \in \zmath} : \forall i \in \zmath, h_{i} \in \nmath
\right\}.
\end{align*}
As for the $k$-parent WFSB metapopulation process, each patch is associated to a type ($0$ or $1$) and an age, but now they have a different interpretation. Indeed, in the SPOM, a "type~$0$" patch corresponds to a patch which cannot contain nonexpired type~$1$ seeds, while a "type~$1$" patch is a patch which can potentially contain type~$1$ seeds, the age $h_{i}$ encoding the last time type~$1$ seeds could have entered the seed bank.

\begin{defn}(BOA process) Let $(O,h) \in \fcal^{\infty} \times \hcal^{\infty}$. The Best Occupancy Achievable process (or BOA process) with parameters $(H,p)$ and with initial conditions $(O,h)$ is the $\left(\fcal^{\infty} \times \hcal^{\infty}\right)$-valued Markov process
$(O^{\infty,n},h^{\infty,n})_{n \in \nmath}$ defined as follows. 
First, we set $(O^{\infty,0},h^{\infty,0}) = (O,h)$. Then, for all $n \in \nmath$, given $(O^{\infty,n},h^{\infty,n})$~:
\begin{enumerate}
\item Let $(\Ext_{i})_{i \in \zmath}$ be i.i.d $\{0,1\}$-valued random variables such that $\proba(\Ext_{1} = 1) = p$. 
\item For all $i \in \zmath$, if $\Ext_{i} = 0$ and $O_{i}^{\infty,n} \mathds{1}_{\{h_{i}^{\infty,n} \leq H\}} = 1$, then we set
\begin{align*}
O_{i-1}^{\infty,n+1} = O_{i}^{\infty,n+1} = O_{i+1}^{\infty,n+1} &= 1 \\
\text{and }\qquad  h_{i-1}^{\infty,n+1} = h_{i}^{\infty,n+1} = h_{i+1}^{\infty,n+1} &= 0. 
\end{align*}
We do nothing during this step if $\Ext_{i} = 1$ or $O_{i}^{\infty,n} \mathds{1}_{\{h_{i}^{\infty,n} \leq H\}} = 0$. 
\item For all $i \in \zmath$, if $O_{i}^{\infty,n+1}$ was not defined during step~$2$, then we set $O_{i}^{\infty,n+1} = O_{i}^{\infty,n}$ and $h_{i}^{\infty,n+1} = h_{i}^{\infty,n} + 1$.
\end{enumerate}
Moreover, we will say that patch $i \in \zmath$ is reachable at generation $n \in \nmath$ if $O_{i}^{\infty,n} \un{h_{i}^{\infty,n} \leq H} = 1$.
\end{defn}

The BOA process represents all the patches which can potentially contain seeds produced by the ones initially present (as given by $(O,h)$), given the extinction events. In other words, informally, the BOA process keeps track of the 
patches that are linked to the patches originally containing viable seeds by means of a path of reachable patches. Notice that $O_{i}^{\infty,n}$ describes the composition of the seed bank, while extinction events affect the standing vegetation. Therefore, an extinction event affecting patch $i$ during the n-th generation does not set the value of $O_{i}^{\infty,n}$ to $0$.

The BOA process is a best-case scenario, in the sense that using the same extinction events to construct the BOA process and the $k$-parent WFSB metapopulation process, it is possible to couple both processes so that all patches containing seeds in the $k$-parent WFSB metapopulation process are reachable patches in the BOA process. In order to formalize the coupling property, we introduce a new object associated to our metapopulation process, describing whether the seed bank in each patch contains real seeds, or only ghost seeds.

\begin{defn}($k$-parent occupancy process)
Let $(\xi,h) \in \fcal_{M}\times \hcal_{M}$. The $k$-parent occupancy process 
\begin{equation*}
\left(O^{k,n},h^{k,n}\right)_{n \in \nmath}
= \left(\left(O_{i}^{k,n},h_{i}^{k,n}\right)_{i \in \zmath}\right)_{n \in \nmath}
\end{equation*}
associated to the $k$-parent WFSB metapopulation process $(\xi^{n},h^{n})_{n \in \nmath}$ with parameters $(M,H,g,c,p)$ and initial conditions $(\xi,h)$ is defined as follows.

First, for all $i \in \zmath$, we set
\begin{align*}
O_{i}^{k,0} &:= 1 - \prod_{j \in \llbracket 1,M \rrbracket}\left(1-\xi_{i,j}\right)
= \max\{\xi_{i,j} : j \in \llbracket 1,M \rrbracket\} \\
h_{i}^{k,0} &:= \begin{cases}
\min \{
h_{i,j} : j \in \llbracket 1,M \rrbracket \text{ and } \xi_{i,j} = 1
\} &\text{ if } O_{i}^{k,0} = 1 \\
0 &\text{ otherwise}.
\end{cases}
\end{align*}

Then, for all $n \in \nmath^{*}$ and $i \in \zmath$, we set
\begin{align*}
O_{i}^{k,n} &:= 1 - \prod_{j \in \llbracket 1,M \rrbracket}\left(1-\xi_{i,j}^{n}\right) \\
h_{i}^{k,n} &:= \begin{cases}
\min \{
h_{i,j}^{n} : j \in \llbracket 1,M \rrbracket \text{ and } \xi_{i,j}^{n} = 1
\} &\text{ if } O_{i}^{k,n} = 1 \\
h_{i}^{k,n-1} + 1 &\text{ otherwise}.
\end{cases}
\end{align*}
\end{defn}

In this setting, $O_{i}^{k,n} = 1$ if and only if at the beginning of generation~$n$, before germination occurs, the patch $i$ contains at least one (potentially expired) seed which was initially of type~$1$. In this case, $h_{i}^{k,n}$ is the number of complete generations spent in the seed bank by the youngest of such seeds. Therefore, patch $i$ contains at least one type $1$ seed at the beginning of generation $n$ if, and only if it contains seeds which where initially of type $1$ (i.e, $O_{i}^{k,n} = 1$) and the youngest seeds among these ones entered the seed bank at most $H+1$ generations ago (i.e, $h_{i}^{k,n} \leq H$, as $h_{i}^{k,n} = 0$ if the seeds entered the seed bank during the previous generation), that is, if and only if
$O_{i}^{k,n} \mathds{1}_{\{h_{i}^{k,n} \leq H \}} = 1$.

\begin{rem}
Note that the $k$-parent occupancy process is also defined on the state space $\fcal^{\infty} \times \hcal^{\infty}$. However, contrary to the BOA process, the $k$-parent occupancy process \textit{cannot} be considered as a SPOM, since $(O^{k,n+1},h^{k,n+1})$ does not depend only on $(O^{k,n},h^{k,n})$. Therefore, both processes are intrinsically different.
\end{rem}

In all that follows, we will say that the BOA process $(O^{\infty,n},h^{\infty,n})_{n \in \nmath}$ \textit{associated to} the $k$-parent WFSB metapopulation process with parameters $(M,H,g,c,p)$ and initial condition $(\xi,h)$ is the BOA process with parameters $(H,p)$ and initial condition $(O^{k,0},h^{k,0})$, constructed using the same extinction events as the $k$-parent WFSB metapopulation process. Under this coupling, the $k$-parent WFSB metapopulation process and its associated BOA process satisfy the following relation:
\begin{equation*}
\forall n \in \nmath, \forall i \in \zmath, O_{i}^{k,n} \leq O_{i}^{\infty,n} \text{ and } h_{i}^{k,n} \geq h_{i}^{\infty,n}.
\end{equation*}
This result will be proved in Section \ref{sec:BOA_process}.

\subsubsection{Convergence of the $k$-parent occupancy process to the BOA process}

When $M$ and $k$ are finite, deviations from the BOA process can occur in the following three cases:
\begin{enumerate}
\item Type $1$ plants are present in a patch, but none of them is chosen as a potential parent.
\item Non-expired type $1$ seeds are present in a patch, but none of them germinate.
\item Several type $1$ seeds entered the seed bank less than $H+1$ generations ago, but all of them already germinated, and there is no remaining non-expired type $1$ seeds in the seed bank.
\end{enumerate}

However, when both $M \to + \infty$ and $k \to + \infty$ in an appropriate way, we can show that the occupancy process converges to the BOA process. For this convergence to occur, two conditions need to be satisfied. First, $k$ needs to grow to $+ \infty$ "faster" than $M$. We will set $k = \lceil M \rceil^{\alpha}$, with $\alpha > 1$, and hence define a sequence of $\lceil M \rceil ^{\alpha}$-parent WFSB processes. Notice that since the $k$ potential parents of an individual do not have to be necessarily different, it is possible to have $k > 3\lfloor gM \rfloor$ (the number of plants in the focal patch and the two neighbouring patches). 
Then, we will need the following constraints on the initial conditions of the processes. Let $(O^{\infty},h^{\infty}) \in \left(\fcal^{\infty} \times \hcal^{\infty}\right)$ satisfying
\begin{equation}\label{eqn:condition_ci_boa} 
\forall i \in \zmath, \text{ if } O_{i}^{\infty} = 0, \text{ then } h_{i}^{\infty} = 0
\end{equation}
encode which patches are initially occupied, and what is the age of the youngest type $1$ seeds in each of these patches. By convention, we set $h_{i}^{\infty} = 0$ for the patches initially empty. For all $i \in \zmath$ such that $O_{i}^{\infty} = 1$, let $g_{i} \in (0,1]$, which will represent the proportion of youngest type $1$ seeds in patch $i$. 

For all $M \geq 2$, let $(\xi^{(M)},h^{(M)}) \in (\fcal_{M} \times \hcal_{M})$ be such that for all $i \in \zmath$ and $j \in \llbracket 1,M \rrbracket$, the following conditions are satisfied. 
\begin{align*}
\text{(A) }& \text{If } O_{i}^{\infty} = 0, \text{ then } \xi_{i,j}^{(M)} = 0. \\
\text{(B) }& \text{If } O_{i}^{\infty} = 1 \text{ and } h_{i,j}^{(M)} < h_{i}^{\infty}, \text{ then } \xi_{i,j}^{(M)} = 0.\\
\text{(C) }& \text{If } O_{i}^{\infty} = 1, \text{ then } M^{-1}\sum_{i = 1}^{M} \xi_{i,j}^{(M)} \un{h_{i,j}^{(M)} = h_{i}^{\infty}} = M^{-1}\lfloor g_{i} M \rfloor.
\end{align*}
Intuitively, conditions (A), (B) and (C) ensure that the patches initially occupied in all $k$-parent WFSB metapopulation processes are the same. Conditions (B) and (C) implies that in each patch, the youngest type~$1$ seeds (if present) have the same age for all processes. Moreover, condition (C) quantifies the proportion of youngest type~$1$ seeds in the seed bank, and ensures they represent a non-negligible portion of the seed bank, even when $M \to + \infty$. Note that this constraint is on the proportion of the \textit{youngest} type~$1$ seeds, and not on the proportion of all type~$1$ seeds. 

\begin{thm}\label{thm:cvg_spom}
Let $\alpha > 1$. For all $M \geq 2$, let $\left(O^{(M),n},h^{(M),n}\right)_{n \in \nmath}$ be the $\lceil M^{\alpha} \rceil$-parent occupancy process associated to the $\lceil M^{\alpha} \rceil$-parent WFSB metapopulation process with parameters $(M,H,g,c,p)$ and initial condition $(\xi^{(M)},h^{(M)})$, and let $\left(O^{(M),\infty,n},h^{(M),\infty,n}\right)_{n \in \nmath}$ be the BOA process associated to the same WFSB metapopulation process. Then, for all $N \in \nmath$, 
\begin{equation*}
\proba\left(\bigcap_{n = 0}^{N}
\left(\left\{
\forall i \in \zmath, O_{i}^{(M),n} = O_{i}^{(M),\infty,n}
\right\} \cap \left\{
\forall i \in \zmath, h_{i}^{(M),n} = h_{i}^{(M),\infty,n}
\right\}\right)
\right) 
\xrightarrow[M \to + \infty]{}
1.
\end{equation*}
\end{thm}

One of the biological interpretations of this result is that in the limit considered, the metapopulation dynamics is well approximated by the BOA process. Moreover, this theorem bridges the gap between individual-based metapopulation models and SPOMs, in the sense that the BOA process is the limit of the $k$-parent WFSB metapopulation process under a suitable scaling. 

\begin{sketch}
The proof is structured as follows. 
%First, by condition~(A), type~$1$ seeds are initially present in only a finite number of patches. Since a plant can only send seeds in the neighbouring patches, the number of occupied patches stays finite. 
First, we show that if the number of occupied patches is finite initially (as assumed in Condition~(A)), then it remains finite. The proof relies on the observation that colonization is only possible towards or from neighbouring patches. 

Then, we show that when $k = \lceil M^{\alpha} \rceil$ is large enough, in each patch, the $k$ potential parents chosen to refill any seed bank compartment in the patch span all the $3\lfloor gM \rfloor$ possible potential parents with high probability. In this case, the $\lfloor gM \rfloor$ new seeds entering the seed bank of the patch are all of the same type. Therefore, we can distinguish two periods:
\begin{itemize}
\item A transition period corresponding to the $H+1$ first generations, during which some seeds initially present are still in the seed bank and viable, implying that some viable seeds of the same age (and in the same patch) are potentially of different types;
\item What we will call the "post-transition period", corresponding to the other generations, during which all the seeds initially present and still in the seed bank are no longer viable, and during which all viable seeds of the same age and in the same patch are of the same type (with an high probability). 
\end{itemize}
We first focus on the proportion of viable seeds of a given age in the seed bank after the transition period. We show that for all $0 \leq h \leq H$, the proportion of age $h$ seeds (which are generally all of the same type) is roughly equal to $g(1-g)^{h}$, and that approximately $g^{2}(1-g)^{h}M$ ($> 1$ for $M$ large enough) such seeds germinate during the next generation. Notice that it is what we would obtain if each seed germinated independently from others and with probability $g$. In particular, if type~$1$ seeds enter the seed bank of a patch during generation~$n$, then with high probability at least one of them will germinate during each of the generations $n+1$, ..., $n+H+1$, and produce new seeds if the patch is not affected by an extinction event, as in the BOA process. 

We conclude by considering the transition period. At the beginning of generation $n \in \llbracket 0,H \rrbracket$, in each non-empty patch:
\begin{itemize}
\item Either no new type~$1$ seeds have already entered the seed bank. By conditions (B) and (C), the proportion of youngest type~$1$ seeds is then roughly equal to $g_{i}(1-g)^{n}$, and approximately $g g_{i} (1-g)^{n}M$ ($> 1$ for $M$ large enough) such seeds germinate during the next generation.
\item Either new type~$1$ seeds have already entered the seed bank. We are then in the same situation as during the post-transition period. 
\end{itemize}
\end{sketch}

\subsubsection{Critical patch extinction probability}

Using the coupling with the BOA process, we will also show the existence of a critical patch extinction probability $p_{crit}(H)$ depending only on $H$ such that for all $p > p_{crit}(H)$, the metapopulation will almost surely become empty in finite time no matter the values of $M$, $g$, $c$ or $k$.

\begin{thm}\label{thm:proba_critique}
For all $H \in \nmath$, there exists $p_{crit}(H) \in (0,1)$ such that for all $M \in \nmath^{*}$, $k \in \nmath \backslash \{0,1\}$, $g \in (0,1)$ and $c \in (0,1/2)$, for all $(\xi,h) \in \fcal_{M} \times \hcal_{M}$ and $p > p_{crit}(H)$, if $(O^{k,n},h^{k,n})_{n \in \nmath}$ is the $k$-parent occupancy process associated to the $k$-parent WFSB metapopulation process with parameters $(M,H,g,c,p)$ and initial condition $(\xi,h)$, then
\begin{equation*}
\lim\limits_{n \to + \infty} \proba\left(\forall i \in \zmath, O_{i}^{k,n}\mathds{1}_{\{h_{i}^{k,n} \leq H\}} = 0\right) = 1.
\end{equation*}
\end{thm}
The proof of this result, which can be found in Section~\ref{sec:ext_threshold}, relies on the coupling between the $k$-parent WFSB metapopulation process and the BOA process, together with appropriate results in percolation theory.

\begin{figure}[ht]
\centering
\includegraphics[width = 0.65\linewidth]{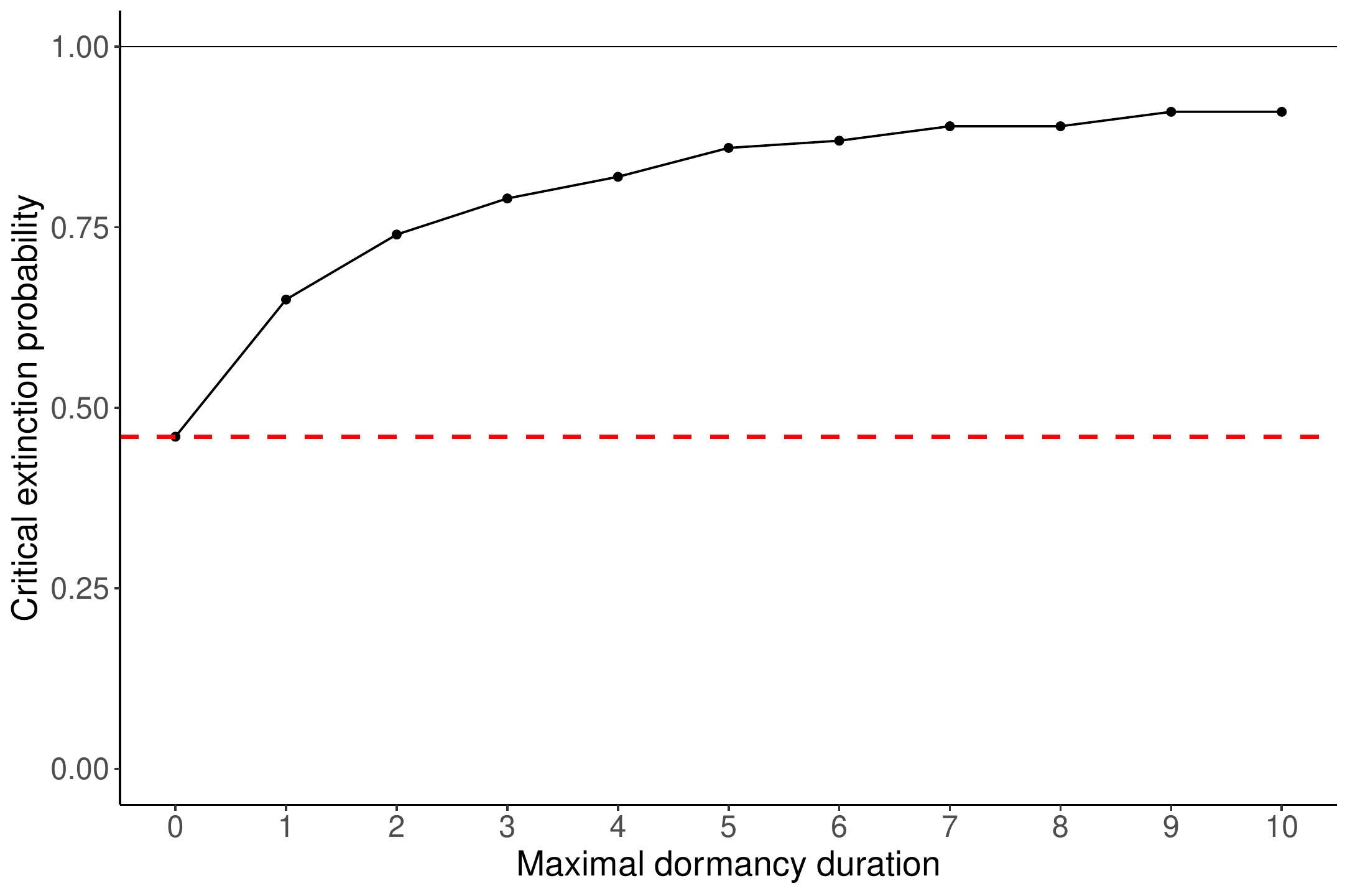}
\caption{Approximate value of the critical extinction probability $p_{crit}(H)$ as a function of the maximal dormancy duration $H$. The red dashed line indicates $p_{crit}(0)$, or in other words, the extinction probability above which (real) plants persistence without a seed bank is not possible. The continuous line indicates $p_{crit}(\infty) = 1$ (see Remark \ref{rem:limiting_case}). See the Appendix for details on the method used to compute $p_{crit}(H)$.}\label{fig:extinction_probability}
\end{figure}

\begin{sketch}
The main idea behind the proof is the following. First, we use the coupling between the $k$-parent WFSB metapopulation process and the BOA process introduced earlier. This coupling ensures that each occupied patch in the $k$-parent WFSB metapopulation process is reachable in the BOA process. We can then focus on the study of the simpler BOA process, since the fact that no sites are reachable in the BOA process will then guarantee that the $k$-parent WFSB metapopulation process has become empty. Using results from percolation theory, we obtain the existence of a critical extinction probability $p_{crit}(H)$ for the BOA process, such that:
\begin{itemize}
\item For all $p > p_{crit}(H)$, the BOA process goes extinct in finite time almost surely ;
\item For all $p < p_{crit}(H)$, the probability that the BOA process goes extinct in finite time is (strictly) less than $1$. 
\end{itemize}
We conclude by using the fact that the $k$-parent WFSB metapopulation process is embedded in the BOA process. 
\end{sketch}

The biological interpretation of this theorem is the following. For each maximal dormancy duration $H \in \nmath$, there exists a critical extinction probability $p_{crit}(H)$ above which any metapopulation evolving according to a $k$-parent WFSB metapopulation process of maximal dormancy duration $H$ will almost surely go extinct in finite time, no matter how quickly plants can invade a patch initially empty (which is quantified by $k$ and to a lesser extent $c$). In particular, no metapopulation without a seed bank can persist if the patch extinction probability is above $p_{crit}(0)$. $p_{crit}(H)$ is increasing with $H$, so the ability to form a seed bank can potentially allow population persistence and expansion in highly disturbed fragmented environments. See Figure \ref{fig:extinction_probability} for approximate values for $p_{crit}(H)$, computed using the method presented in the Appendix. Numerical simulations show the existence of parameter sets $(M,H,g,c,p)$ with $H > 0$ and $p > p_{crit}(0)$ for which population persistence is indeed possible (see Figure \ref{fig:invasion_persistence}). Since the $k$-parent occupancy process converges to the BOA process, which does not go extinct with positive probability if $p < p_{crit}(H)$, the critical extinction probability $p_{crit}(H)$ we obtain is optimal, in the sense that it is not possible to obtain a better upper bound on the critical extinction probability for the $k$-parent WFSB metapopulation process which depends only on $H$, and not also on one of the other parameters. 

\begin{rem}\label{rem:limiting_case}
In the limiting case $H = \infty$, seeds never expire, and the number of reachable patches in the BOA process cannot decrease, implying that $p_{crit}(\infty) = 1$. 
\end{rem}

\begin{figure}[ht]
\centering
\begin{subfigure}[b]{0.48\textwidth}
\includegraphics[width = \linewidth]{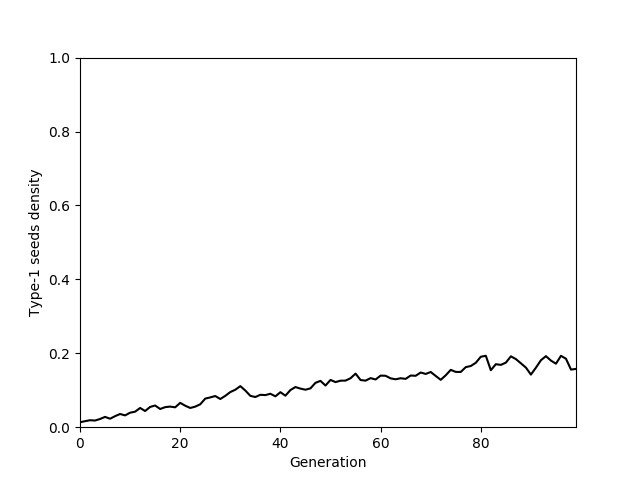}
\subcaption{p = 0.5 and H = 1}
\end{subfigure}
\hfill
\begin{subfigure}[b]{0.48\textwidth}
\includegraphics[width = \linewidth]{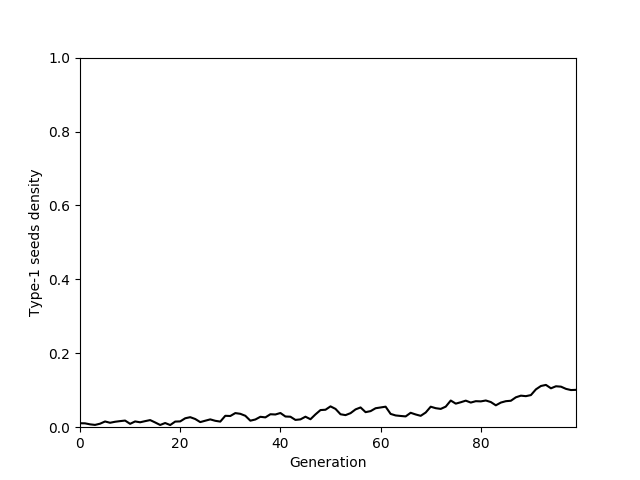}
\subcaption{p = 0.7 and H = 3}
\end{subfigure}
\caption{Plant metapopulation expansion for extinction probabilities $p_{crit}(0) < p < p_{crit}(H)$, and for a maximal dormancy duration $H \neq 0$. The values taken by the other parameters are $M = 100$, $g = 0.5$, $c = 0.05$ and $k = 25$. Initially, $5$ consecutive patches contained $gM = 50$ type $1$ seeds, and all the other seed bank compartments were empty. Since only the first $100$ generations were considered, the simulation was performed on a torus of $200$ patches, and the density of type $1$ seeds was computed over these $200$ patches.
}\label{fig:invasion_persistence}
\end{figure}

\section{Proof of the convergence of the $k$-parent occupancy process to the BOA process}\label{sec:cvg_to_BOA}
The goal of this section is to show that the $k$-parent occupancy process converges to the BOA process in the sense of Theorem \ref{thm:cvg_spom}, that is, when both $M \to + \infty$ and $k \to + \infty$, but with $k$ increasing "faster" than $M$. 
In order to do so, we first focus on the post-transition period, and take an initial condition allowing us to skip the transition period. Then, we explain how to adapt the proof to take into account more general initial conditions. 

First, we set some notation. We use the notation from Theorem \ref{thm:cvg_spom} throughout this section. For all $M \geq 2$ and $h \in \nmath$, let $(\xi^{(M),n},h^{(M),n})_{n \in \nmath}$ be the $\lceil M^{\alpha} \rceil$-parent WFSB metapopulation process with parameters $(M,H,g,c,p)$ associated to the $\lceil M^{\alpha} \rceil$-parent occupancy process $(O^{(M),n},h^{(M),n})_{n \in \nmath}$. For every $n \in \nmath$, $i \in \zmath$ and $h \in \nmath$, let
\begin{equation*}
E_{i,h}^{(M),n} := \{
j \in \llbracket 1,M \rrbracket : h_{i,j}^{(M),n} = h
\}
\end{equation*}
be the set of all seed bank compartments in patch $i$ containing seeds of age $h$ at the beginning of generation $n$, and let $G_{i}^{(M),n}$ be the set of all seed bank compartments in patch $i$ which contain the seeds germinating at the beginning of generation $n$. 

Let $(a_{n})_{n \in \nmath}$ be the sequence defined by
\begin{equation*}
a_{0} = 0 \text{ and } \forall n \in \nmath, a_{n+1} = 3a_{n}+1,
\end{equation*}
and for all $0 < \epsilon < 1$, let 
\begin{equation*}
W_{h}^{(M),\epsilon} := \left[
\lfloor gM \rfloor \left(\left(
1-\frac{\lfloor gM \rfloor}{M}
\right)^{h} -\epsilon a_{h}\right),
\lfloor gM \rfloor \left(\left(
1-\frac{\lfloor gM \rfloor}{M}
\right)^{h} +\epsilon a_{h}\right)
\right].
\end{equation*}
Recalling that $(O^{(M),\infty,n},h^{(M),\infty,n})_{n \in \nmath}$ is the BOA process associated to $(\xi^{(M),n},h^{(M),n})_{n \in \nmath}$, we introduce the event
\begin{equation*}
\dboa_{i}^{(M),n} := \left\{
O_{i}^{(M),n} \neq O_{i}^{(M),\infty,n} \text{ or } h_{i}^{(M),n} \neq h_{i}^{(M),\infty,n}
\right\}.
\end{equation*}
Given the initial condition for the associated BOA process, we already know that for $n = 0$ the BOA process associated to $(\xi^{(M),n},h^{(M),n})_{n \in \nmath}$ is equal to the occupancy process associated to the same process. Therefore, 
\begin{equation*}
\proba\left(\bigcap_{i \in \zmath}\overline{\dboa_{i}}^{(M),0}\right) = 1. 
\end{equation*}
Moreover, as only neighbouring sites can send colonizing seeds, if type~$1$ seeds are initially in a finite number of patches, then it is also the case after any arbitrary finite duration. More specifically, if we set
\begin{align*}
i_{min}^{0} &:= \min\{i \in \zmath : O_{i}^{\infty} = 1\} 
\text{ and } i_{max}^{0} := \max\{i \in \zmath : O_{i}^{\infty} = 1\}, \\
\intertext{and if for all $n \in \nmath$, we set}
i_{min}^{n+1} &:= i_{min}^{n}-1 
\text{ and } i_{max}^{n+1} := i_{max}^{n}+1,
\end{align*}
then the only patches which can potentially contain type~$1$ seeds after $n$ generations are the patches with index $i \in \llbracket i_{min}^{n},i_{max}^{n} \rrbracket$. In other words, for all $M \geq 2$, $n \in \nmath$ and $i \in \zmath \backslash \llbracket i_{min}^{n},i_{max}^{n} \rrbracket$, 
we have
$O_{i}^{(M),n} = O_{i}^{(M),\infty,n} = 0 \text{ and }h_{i}^{(M),n} = h_{i}^{\infty,n}$, 
so that
$\overline{\dboa}_{i}^{(M),n}$ holds a.s.
Therefore, for all $N \in \nmath$, 
\begin{equation}\label{eqn:equality_theorem_2}
\proba\left(
\bigcup_{n = 0}^{N} \bigcup_{i \in \zmath} \dboa_{i}^{(M),n}
\right)
= \proba\left(
\bigcup_{n = 1}^{N} \bigcup_{i \in \llbracket i_{min}^{n},i_{max}^{n} \rrbracket} \dboa_{i}^{(M),n}
\right).
\end{equation}
In all that follows, in order to ease notation, for all $n \in \nmath$, we set $I_{n} = \llbracket i_{min}^{n},i_{max}^{n} \rrbracket$

Let $M \geq 2$ such that $\lfloor gM \rfloor > 1$, and $N \geq 1$. We recall that deviations from the BOA process can occur in the following cases.
\begin{enumerate}
\item Type $1$ plants are present in a patch, but are never chosen as potential parents.
\item The seed bank contains non-expired type $1$ seeds, but none of them germinate during the generation we consider. 
\item Some type $1$ seeds did enter the seed bank less than $H+1$ generations ago, but all of them already germinated, and the seed bank is now empty. 
\end{enumerate}

If the initial condition $(\xi^{(M),0},h^{(M),0})$ satisfies the extra hypothesis:
\begin{itemize}
\item (IC1) All the viable seeds of the same age and in the same patch are of the same type, or in formula,
\begin{equation*}
\forall i \in \zmath, \forall j_{1}, j_{2} \in \llbracket 1,M \rrbracket, h_{i,j_{1}}^{(M),0} = h_{i,j_{2}}^{(M),0} \leq H \Longrightarrow \xi_{i,j_{1}}^{(M),0} = \xi_{i,j_{2}}^{(M),0};
\end{equation*}
\item (IC2) For all $h \in \llbracket 0,H \rrbracket$, the proportion of age $h$ seeds is roughly equal to $g(1-g)^{h}$, in the sense that
\begin{equation*}
\Card(E_{i,h}^{(M),0}) \in W_{h}^{(M),\epsilon}
\end{equation*}
\end{itemize}
for $\epsilon$ appropriately chosen, then the three deviation cases are covered by the following more general cases:
\begin{itemize}
\item (D1) There exists $n \in \nmath$ and $i \in I_{n+1}$ such that one of the plants in patches $\{i-1, i, i+1\}$ is not chosen as a potential parent by at least one seed bank compartment in patch $i$ during generation $n$.
\item (D2) There exists $n \in \nmath$ and $i \in I_{n}$ such that $h_{i}^{(M),n} \leq H$ but no seed of age $h_{i}^{(M),n}$ germinates during generation~$n$. 
\item (D3) There exists $n \in \nmath$ and $i \in I_{n}$ such that $h_{i}^{(M),n} \leq H$ but $\Card\left(E_{i,h_{i}^{(M),n}}^{(M),n}\right) \notin W_{h_{i}^{(M),n}}^{(M),\epsilon}$.
\end{itemize}
Indeed, initially, all the viable seeds of the same age have the same type, and this stays true until (D1) occurs. If (D1) did not occur yet before generation~$n$, then for all $i \in I_{n}$, all seeds of age $h_{i}^{(M),n}$ are of the same type. They are not necessarily type~$1$ seeds, but if the patch is not empty, then they are all type~$1$ seeds. Therefore, if the patch is not empty and if $\epsilon$ is such that
\begin{equation*}
\lgmr \left(\left(
1 - \lgmr M^{-1}
\right)^{h_{i}^{(M),n}} - \epsilon a_{h_{i}^{(M),n}}\right) > 1,
\end{equation*}
%\begin{align*}
%\lgmr \left(\left(
%1 - \lgmr M^{-1}
%\right)^{h_{i}^{(M),n}} - \epsilon a_{h_{i}^{(M),n}}\right) &> 0 \\
%\text{and }
%\lgmr \left(
%1 - \lgmr M^{-1}
%\right)^{h_{i}^{(M),n}} + \epsilon a_{h_{i}^{(M),n}} &< 1, 
%\end{align*}
then
\begin{equation*}
\Card\left(E_{i,h_{i}^{(M),n}}^{(M),n}\right) \in W_{h_{i}^{(M),n}}^{(M),\epsilon}
\end{equation*}
implies that the patch contains type~$1$ seeds which entered less than $H+1$ generations ago and did not germinate yet, and 
\begin{equation*}
\Card\left(
E_{i,h_{i}^{(M),n}}^{(M),n} \cap G_{i}^{(M),n}
\right) > 0
\end{equation*}
implies that at least one type~$1$ seed germinates during generation~$n$. 

Formally, let $\epsilon > 0$ be such that for all $M \geq 2$ such that  
$\left(1-\lgmr M^{-1}\right)^{H}-\lgmr^{-1} > 0$, 
\begin{align*}
\epsilon &< a_{H}^{-1}\left(\left(
1 - \lgmr M^{-1}
\right)^{H}-\lgmr^{-1}\right). 
%\text{and } \forall h \in \llbracket 0,H \rrbracket, \epsilon &< a_{h}^{-1} \left(
%1 - \left(
%1-\lgmr M^{-1}
%\right)^{h}
%\right). 
\end{align*}
%These two conditions ensure that for all $h \in \llbracket 0,H \rrbracket$, $W_{h}^{(M),\epsilon} \varsubsetneq [0,1]$. In other words, if 
%\begin{equation*}
%\Card\left(E_{i,h}^{(M),n}\right) \in W_{h}^{(M),\epsilon},
%\end{equation*}
%then the patch contains seeds of age~$h$, but not exclusively. 
For all $n \in \nmath$ and $i \in \zmath$, we define the following events corresponding respectively to (D1), (D2) and (D3):
\begin{align*}
\errpar_{i}^{(M),n} &:= \left\{\text{There exists a plant in patches }\{i-1,i,i+1\}\text{ which is not chosen as a} \right. \\
& \quad \quad \text{ potential parent by at least one seed bank compartment in patch i during} \\
& \quad \quad \left.\text{ generation n}\right\}, \\
\errg_{i}^{(M),n} &:= \left\{h_{i}^{(M),n} \leq H\text{ and }\Card\left(
E_{i,h_{i}^{(M),n}}^{(M),n} \cap G_{i}^{(M),n}
\right) = 0\right\}, \\
\text{and } \quad  \errf_{i}^{(M),n} &:= \left\{h_{i}^{(M),n} \leq H\text{ and }\Card\left(
E_{i,h_{i}^{(M),n}}^{(M),n}
\right) \notin W_{h_{i}^{(M),n}}^{(M),\epsilon}\right\}.
\end{align*}
We assume that the initial condition $(\xi^{(M),0}, h^{(M),0})$ satisfies conditions (IC1) and (IC2). Then, 
\begin{equation}\label{eqn:equality_theorem_1}
\proba\left(
\bigcup_{n = 1}^{N} \bigcup_{i \in I_{n}} \dboa_{i}^{n} 
\right) \leq 
\proba \left(
\bigcup_{n = 0}^{N-1} \bigcup_{i \in I_{n+1}} 
\errpar_{i}^{(M),n} \cup \errg_{i}^{(M),n} \cup \errf_{i}^{(M),n+1}
\right).
\end{equation}

\begin{rem}
If $(\xi^{(M),0},h^{(M),0})$ does not satisfy condition (IC1) or (IC2), then we show at the end of the section that after a transition period of $H+1$ generations, $(\xi^{(M),H+1},h^{(M),H+1})$ satisfies conditions (IC1) and (IC2) with high probability. We can then restart the process from $(\xi^{(M),H+1},h^{(M),H+1})$ and conclude. 
\end{rem}

In order to shorten the notation, let $\errgf_{i}^{(M),n}$ be the event defined as
\begin{equation*}
\errgf_{i}^{(M),n} :=  \errg_{i}^{(M),n} \cup \errf_{i}^{(M),n+1}.
\end{equation*}
Since the choice of the potential parents does not depend on their types and is independent from one patch (or generation) to another, we have
\begin{align}
&\proba \left(
\bigcup_{n = 0}^{N-1} \bigcup_{i \in I_{n+1}} 
\errpar_{i}^{(M),n} \cup \errgf_{i}^{(M),n}
\right) \nonumber \\ 
&\leq \sum_{n = 0}^{N-1} \sum_{i \in I_{n+1}}
\proba\left(
\errpar_{i}^{(M),n}
\right) + \proba\left(\left.
\bigcup_{n = 0}^{N-1} \bigcup_{i \in I_{n+1} }
\errgf_{i}^{(M),n} \right|
%\bigcap_{\substack{0 \leq n' \leq N-1 \\ i \in I_{n'+1}}} 
\left(\overline{\errpar}_{i}^{(M),n'}\right)_{\substack{0 \leq n' \leq N-1 \\ i \in I_{n'+1}}} 
\right) \\ \nonumber
&\leq \sum_{n = 0}^{N-1} \left(
i_{max}^{0} - i_{min}^{0} +1 +2(n+1)
\right)\proba\left(\errpar_{i_{min}^{0}}^{(M),0}\right) 
 + \proba\left(
\left. \bigcup_{i \in I_{1}} \errgf_{i}^{(M),0} \right|
%\bigcap_{\substack{0 \leq n' \leq N-1 \\ i \in I_{n'+1}}} 
\left(\overline{\errpar}_{i}^{(M),n'}\right)_{\substack{0 \leq n' \leq N-1 \\ i \in I_{n'+1}}}
\right) \\ \nonumber
& \qquad + \sum_{n = 1}^{N-1} \proba\left(
\left. \bigcup_{i \in I_{n+1}} \errgf_{i}^{(M),n}  \right|
\left(
%\bigcap_{\substack{0 \leq n' \leq N-1 \\ i \in I_{n'+1}}} 
\overline{\errpar}_{i}^{(M),n'} \right)_{\substack{0 \leq n' \leq N-1 \\ i \in I_{n'+1}}}
\cap \left(
%\bigcap_{\substack{0 \leq n' \leq n-1 \\ i \in I_{n'+1}}} 
\overline{\errgf}_{i}^{(M),n'} 
\right)_{\substack{0 \leq n' \leq n-1 \\ i \in I_{n'+1}}}
\right) \\
&\leq \left[
N\left(i_{max}^{0}-i_{min}^{0}+1\right) + N(N+1)
\right] \proba\left(\errpar_{i_{min}^{0}}^{(M),0}\right) \label{eqn:terme_moche_1} \\
& \qquad + \proba\left(
\left. \bigcup_{i \in I_{1}} \errgf_{i}^{(M),0} \right|
%\bigcap_{\substack{0 \leq n' \leq N-1 \\ i \in I_{n'+1}}} 
\left(\overline{\errpar}_{i}^{(M),n'}\right)_{\substack{0 \leq n' \leq N-1 \\ i \in I_{n'+1}}}
\right) \label{eqn:terme_moche_2}\\ 
& \qquad + \sum_{n = 1}^{N} \proba\left(
\left. \bigcup_{i \in I_{n+1}} \errgf_{i}^{(M),n}  \right|
\left(
%\bigcap_{\substack{0 \leq n' \leq N-1 \\ i \in I_{n'+1}}} 
\overline{\errpar}_{i}^{(M),n'} \right)_{\substack{0 \leq n' \leq N-1 \\ i \in I_{n'+1}}}
\cap \left(
%\bigcap_{\substack{0 \leq n' \leq n-1 \\ i \in I_{n'+1}}} 
\overline{\errgf}_{i}^{(M),n'} 
\right)_{\substack{0 \leq n' \leq n-1 \\ i \in I_{n'+1}}}
\right). \label{eqn:terme_moche_3}
\end{align} 
Bounding the quantity in (\ref{eqn:terme_moche_1})  from above is straightforward, and is carried out in Section~\ref{sec:section_2_1}. The main obstacle to the study of the two other terms stems from the fact that all the events considered are linked together by $(h^{(M),n})_{n \in \nmath}$. However, we can circumvent this problem by working conditionally on $h^{(M),n}$. Indeed, given $h_{i}^{(M),n}$, since $h_{i}^{(M),n+1}$ can only be equal to $0$ or $h_{i}^{(M),n}+1$,
\begin{align*}
\errf_{i}^{n+1} &:= \left\{
\Card\left(
E_{i,h_{i}^{(M),n+1}}^{(M),n+1}
\right) \notin W_{h_{i}^{(M),n+1}}^{(M),\epsilon} \text{ and } h_{i}^{(M),n+1} \leq H
\right\} \\
&\subseteq \left\{
\Card\left(
E_{i,0}^{(M),n+1}
\right) \notin W_{0}^{(M),\epsilon} \text{ and } 0 \leq H
\right\} \\
& \quad 
\cup
\left\{
\Card\left(
E_{i,h_{i}^{(M),n}+1}^{(M),n+1}
\right) \notin W_{h_{i}^{(M),n}+1}^{(M),\epsilon} \text{ and } h_{i}^{(M),n}+1 \leq H
\right\} \\
&= \left\{
\Card\left(
E_{i,h_{i}^{(M),n}+1}^{(M),n+1}
\right) \notin W_{h_{i}^{(M),n}+1}^{(M),\epsilon} \text{ and } h_{i}^{(M),n}+1 \leq H
\right\}
\end{align*}
since $W_{0}^{(M),\epsilon} = \{\lfloor gM \rfloor\} = \{\Card(E_{i,0}^{(M),n+1})\}$. 
Therefore, 
\begin{equation} \label{eqn:reecriture_event_conditional}
\errgf_{i}^{n} \subseteq \{h_{i}^{(M),n} \leq H\} \cap \left(
\left\{
\Card\left(
E_{i,h_{i}^{(M),n}}^{(M),n}
\cap G_{i}^{(M),n} \right) = 0
\right\}
\cup
\left\{
\mathrm{Card}\left(
E_{i,h_{i}^{(M),n}+1}^{(M),n+1}
\right) \notin W_{h_{i}^{(M),n}+1}^{(M),\epsilon}
\right\}
\right).
\end{equation}
In order to use (\ref{eqn:reecriture_event_conditional}), in Section~\ref{sec:section_2_2}, we establish an upper bound on 
\begin{equation}\label{proba:truc_moche_2_et_3}
\proba\left(\left.\left\{
\Card\left(E_{i,h}^{(M),n} \cap G_{i}^{(M),n}\right) = 0
\right\} \cup 
\left\{
\Card\left(E_{i,h+1}^{(M),n+1}\right) \notin W_{h+1}^{(M),\epsilon}
\right\}
\right|\left\{
\Card\left(E_{i,h}^{(M),n}\right) \in W_{h}^{(M),\epsilon}
\right\}
\right)
\end{equation}
for all $h \in \llbracket 0,H \rrbracket$, which does not depend on the value of $h$. In Section~\ref{sec:section_2_3}, we use it, combined with the upper bound on (\ref{eqn:terme_moche_1}) from Section~\ref{sec:section_2_1}, in order to complete the proof of Theorem~\ref{thm:cvg_spom}. 

\subsection{Upper bound on $\proba(\mathrm{Par}^{(M),n}_{i})$}\label{sec:section_2_1}
We set $c^{*} = \min(c,1-2c)$. 
The goal of this section is to show the following lemma.

\begin{lem}\label{lem:maj_p_rm}
For all $M \geq 2$, $n \in \nmath$ and $i \in \zmath$,
\begin{equation*}
\proba\left(\mathrm{Par}^{(M),n}_{i}\right) \leq 3g^{2}M^{2}
\exp\left(
M^{\alpha} \ln\left(
1 - \frac{c^{*}}{gM}
\right)\right).
\end{equation*}
\end{lem}

A direct consequence of this lemma is the fact that since $\alpha > 1$,
\begin{equation*}
\proba\left(\mathrm{Par}^{(M),n}_{i}\right) \xrightarrow[M \to + \infty]{} 0.
\end{equation*}

\begin{proof}
Assume that $c^{*} = c$. 
Let $\widetilde{\mathrm{Par}}^{(M),n}_{i}$ be the event: $\{$"The first seed which germinated in patch $i+1$ was not chosen as a potential parent by the first seed bank compartment in patch $i$ to be refilled"$\}$. Then,
\begin{equation*}
\proba\left(\mathrm{Par}^{(M),n}_{i}\right) \leq 3 \lgmr^{2} \proba\left(\widetilde{\mathrm{Par}}^{(M),n}_{i}\right).
\end{equation*}
Indeed, $\mathrm{Par}^{(M),n}_{i}$ is the event $\{$"at least one plant in one of the patches $\{i-1$, $i$, $i+1\}$ is not chosen as a potential parent in order to refill at least one seed bank compartment in patch $i$"$\}$. There exists $3 \lgmr ^{2}$ pairs "plant not chosen in patch $i-1$, $i$ or $i+1$/seed bank compartment in patch $i$", and as $c^{*} = c$, plants in patches $i-1$ and $i+1$ have less chances of being chosen as potential parents than plants in patch $i$.

Then, each one of the $\lceil M^{\alpha} \rceil$ potential parents chosen to refill the first seed bank compartment in patch $i$ is \textit{not} the first plant of patch $i+1$ with probability
\begin{equation*}
1 - \frac{c^{*}}{\lgmr} \leq 1 - \frac{c^{*}}{gM}.
\end{equation*}
Hence, 
\begin{align*}
\proba\left(\widetilde{\mathrm{Par}}^{(M),n}_{i}\right) &\leq \left(
1 - \frac{c^{*}}{gM}
\right)^{\lceil M^{\alpha} \rceil} \\
&\leq \left(
1 - \frac{c^{*}}{gM}
\right)^{M^{\alpha}} \\
\text{and } \qquad \proba\left(\mathrm{Par}^{(M),n}_{i}\right) &\leq 
3g^{2}M^{2}
\exp\left(
M^{\alpha} \mathrm{ln}
\left(
1 - \frac{c^{*}}{gM}
\right)\right).
\end{align*}
If $c^{*} \neq c$, then we can directly adapt this proof defining instead the event $\widetilde{\mathrm{Par}}^{(M),n}_{i}$ as the event $\{$"The first seed which germinated in the patch $i$ (instead of the patch $i+1$) was not chosen as a potential parent by the first seed bank compartment in patch $i$ to be refilled"$\}$. 
\end{proof}

\subsection{Upper bound on (\ref{proba:truc_moche_2_et_3})}\label{sec:section_2_2}
In this section, we show an upper bound on 
\begin{equation*}
\proba\left(\left.\left\{
\Card\left(E_{i,h}^{(M),n} \cap G_{i}^{(M),n}\right) = 0
\right\} \cup 
\left\{
\Card\left(E_{i,h+1}^{(M),n+1}\right) \notin W_{h+1}^{(M),\epsilon}
\right\}
\right|\left\{
\Card\left(E_{i,h}^{(M),n}\right) \in W_{h}^{(M),\epsilon}
\right\}
\right)
\end{equation*}
for all $h \in \llbracket 0,H \rrbracket$. In other words, we study the probability that if the number of age $h$ seeds in patch $i$ at the beginning of generation $n$ is roughly equal to $g(1-g)^{h}M$:
\begin{itemize}
\item either no age $h$ seeds germinate during generation $n$
\item or the number of remaining age $h$ seeds (which become age~$h+1$ seeds in the subsequent generation) is significantly different from $g(1-g)^{h+1}M$. 
\end{itemize}
In order to do so, we first observe that if $E \subseteq \llbracket 1,M \rrbracket$ is a non-empty strict subset of $\llbracket 1,M \rrbracket$, then for all $n \in \nmath$ and $i \in \zmath$, $\Card(E \cap G_{i}^{(M),n})$ follows an hypergeometric law. Using the tail inequalities~(10) and (14) from \cite{skala2013hypergeometric} yields the following lemma. 

\begin{lem}\label{lem:hypergeometric_law} 
Let $M \geq 2$ be such that  
$\left(1-\lgmr M^{-1}\right)^{H}-\lgmr^{-1} > 0$.
Let $E \subseteq \llbracket 1,M \rrbracket$ be a non-empty strict subset of $\llbracket 1,M \rrbracket$. Then, for all $n \in \nmath$, $i \in \zmath$ and $0 < \epsilon < 1$, 
\begin{align*}
&\proba\left(
\Card(E \cap G_{i}^{(M),n}) \geq \Card(E) (1+\epsilon) \frac{\lfloor gM \rfloor}{M}
\right) \leq e^{-2\epsilon^{2} \Card(E)^{2} \lfloor gM \rfloor M^{-2}} \\
\text{and } \quad &\proba\left(\Card(E \cap G_{i}^{(M),n}) \leq \Card(E) (1-\epsilon) \frac{\lfloor gM \rfloor}{M}
\right) \leq e^{-2\epsilon^{2} \Card(E)^{2} \lfloor gM \rfloor M^{-2}}
\end{align*}
\end{lem}

Using this lemma, we can obtain the following upper bound. 
\begin{lem}\label{lem:distr_age_of_seeds}
For all $M \geq 2$ such that  
$\left(1-\lgmr M^{-1}\right)^{H}-\lgmr^{-1} > 0$, for all $n \in \nmath$, $i \in \zmath$ and $h \in \llbracket 0,H \rrbracket$, 
\begin{align*}
&\proba\left(\left.\left\{
\Card\left(E_{i,h}^{(M),n} \cap G_{i}^{(M),n}\right) = 0
\right\} \cup 
\left\{
\Card\left(E_{i,h+1}^{(M),n+1}\right) \notin W_{h+1}^{(M),\epsilon}
\right\}
\right|\left\{
\Card\left(E_{i,h}^{(M),n}\right) \in W_{h}^{(M),\epsilon}
\right\}
\right) \\
&\leq 2e^{-2\epsilon^{2} \lfloor gM \rfloor^{3} ((1-\lfloor gM \rfloor M^{-1})^{H}-\epsilon a_{H})^{2} M^{-2}}.
\end{align*}
\end{lem}

\begin{proof}
Let $h \in \llbracket 0,H \rrbracket$, $n \in \nmath$ and $i \in \zmath$. Assume that $\Card(E_{i,h}^{(M),n}) \in W_{h}^{(M),\epsilon}$. We make the following observation. 
\begin{enumerate}
\item If less than $\lfloor gM \rfloor M^{-1} (1+\epsilon) \Card(E_{i,h}^{(M),n})$ age~$h$ seeds germinate during generation~$n$, the number of remaining age~$h$ seeds is bounded from below by
\begin{align*}
& \lfloor gM \rfloor \left[
\left(1-\frac{\lgmr}{M}\right)^{h}-\epsilon a_{h} 
- \frac{\lgmr}{M}(1+\epsilon) \left(\left(1-\frac{\lgmr}{M}\right)^{h} + \epsilon a_{h}\right)
\right] \\
&= \lgmr \left[
\left(1-\frac{\lgmr}{M}\right)^{h} \left(1-\frac{\lgmr}{M}\right) - \epsilon a_{h} - \frac{\lgmr}{M} \epsilon a_{h} - \frac{\lgmr}{M}\epsilon
\left(1 - \frac{\lgmr}{M}\right)^{h} - \frac{\lgmr}{M}\epsilon^{2} a_{h}
\right] \\
&\geq \lgmr \left[
\left(1-\frac{\lgmr}{M}\right)^{h+1} - \epsilon(3a_{h} + 1)
\right] \\
&= \lgmr \left(\left(
1-\frac{\lgmr}{M}
\right)^{h+1} - \epsilon a_{h+1}\right)
\end{align*}
by definition of $(a_{h})_{h \geq 0}$. By Lemma~\ref{lem:hypergeometric_law}, this event happens with probability bounded from below by
\begin{equation*}
1-e^{-2\epsilon^{2} \lfloor gM \rfloor^{3} ((1-\lfloor gM \rfloor M^{-1})^{h}-\epsilon a_{h})^{2} M^{-2}}
\geq 1 - e^{-2\epsilon^{2} \lfloor gM \rfloor^{3} ((1-\lfloor gM \rfloor M^{-1})^{H}-\epsilon a_{H})^{2} M^{-2}}.
\end{equation*}
\item If more than $\lfloor gM \rfloor M^{-1} (1-\epsilon) \Card(E_{i,h}^{(M),n})$ age~$h$ seeds germinate, the number of remaining age~$h$ seeds is bounded from above by
\begin{align*}
& \lfloor gM \rfloor \left[
\left(1-\frac{\lgmr}{M}\right)^{h}+\epsilon a_{h} 
- \frac{\lgmr}{M}(1-\epsilon) \left(\left(1-\frac{\lgmr}{M}\right)^{h} - \epsilon a_{h}\right)
\right] \\
&= \lgmr \left[
\left(1-\frac{\lgmr}{M}\right)^{h} \left(1-\frac{\lgmr}{M}\right) + \epsilon a_{h} + \frac{\lgmr}{M} \epsilon a_{h} + \frac{\lgmr}{M}\epsilon
\left(1 - \frac{\lgmr}{M}\right)^{h} - \frac{\lgmr}{M}\epsilon^{2} a_{h}
\right] \\
&\leq \lgmr \left[
\left(1-\frac{\lgmr}{M}\right)^{h+1} + \epsilon (3a_{h}+1)
\right] \\
&= \lgmr \left(
\left(
1-\frac{\lgmr}{M}
\right)^{h+1} + \epsilon a_{h+1}
\right).
\end{align*}
Again by Lemma~\ref{lem:hypergeometric_law}, this event happens with probability bounded from below by
\begin{equation*}
1 - e^{-2\epsilon^{2} \lfloor gM \rfloor^{3} ((1-\lfloor gM \rfloor M^{-1})^{H}-\epsilon a_{H})^{2} M^{-2}}.
\end{equation*}
\end{enumerate}
Therefore,
\begin{align*}
&\proba\left(\left.\left\{
\Card\left(E_{i,h}^{(M),n} \cap G_{i}^{(M),n}\right) = 0
\right\} \cup 
\left\{
\Card\left(E_{i,h+1}^{(M),n+1}\right) \notin W_{h+1}^{(M),\epsilon}
\right\}
\right|\left\{
\Card\left(E_{i,h}^{(M),n}\right) \in W_{h}^{(M),\epsilon}
\right\}
\right) \\
&\leq \proba\left(\left. \Card\left(E_{i,h}^{(M),n} \cap G_{i}^{(M),n}\right) < \frac{\lfloor gM \rfloor}{M} (1-\epsilon) \Card(E_{i,h}^{(M),n})
\right|\left\{
\Card\left(E_{i,h}^{(M),n}\right) \in W_{h}^{(M),\epsilon}
\right\}
\right) \\
& \quad + \proba\left(\left.\Card\left(E_{i,h}^{(M),n} \cap G_{i}^{(M),n}\right) >
\frac{\lfloor gM \rfloor}{M} (1+\epsilon) \Card(E_{i,h}^{(M),n}) 
\right|\left\{
\Card\left(E_{i,h}^{(M),n}\right) \in W_{h}^{(M),\epsilon}
\right\}
\right) \\
&\leq 2 e^{-2\epsilon^{2} \lfloor gM \rfloor^{3} ((1-\lfloor gM \rfloor M^{-1})^{H}-\epsilon a_{H})^{2} M^{-2}},
\end{align*}
and we can conclude.
\end{proof}

\subsection{Proof of Theorem \ref{thm:cvg_spom}}\label{sec:section_2_3}
In order to show Theorem~\ref{thm:cvg_spom}, we first condition on $h^{(M),n}$ in Eq. (\ref{eqn:terme_moche_3}). In order to do so, we introduce the following notation. For all $n < N \in \nmath$, let $\cond^{(M),n,N}$ be the event defined by
\begin{equation}
\cond^{(M),n,N} := \left(
\overline{\errpar}_{i}^{(M),n'} \right)_{\substack{0 \leq n' \leq N-1 \\ i \in I_{n'+1}}}
\cap \left(
\overline{\errgf}_{i}^{(M),n'} 
\right)_{\substack{0 \leq n' \leq n-1 \\ i \in I_{n'+1}}}.
\end{equation}
Moreover, let $\vcal^{(M),n,N}$ be the set of all possible values for $(h_{i}^{(M),n})_{i \in I_{n+1}}$ given the initial condition and $\cond^{(M),n,N}$. That is, 
\begin{equation*}
\vcal^{(M),n} := \left\{(h_{i})_{i \in I_{n+1}} \in \nmath^{i_{max}^{n+1}-i_{min}^{n+1}+1} : 
\proba\left(\left.
\forall i \in I_{n+1}, h_{i}^{(M),n} = h_{i}
\right|
\cond^{(M),n,N}
\right)
> 0\right\}. 
\end{equation*}

\begin{lem}\label{lem:terme_moche_3}
For all $M \geq 2$ such that  
$\left(1-\lgmr M^{-1}\right)^{H}-\lgmr^{-1} > 0$, 
\begin{align*}
&\sum_{n = 1}^{N} \proba\left(\left.
\bigcup_{i \in I_{n+1}} \errgf_{i}^{(M),n}
\right|\cond^{(M),n,N}\right) \\
&\leq 2
\left(N(i_{max}^{0}-i_{min}^{0} + 3) + N(N+1)\right)
e^{-2\epsilon^{2} \lfloor gM \rfloor^{3} ((1-\lfloor gM \rfloor M^{-1})^{H}-\epsilon a_{H})^{2} M^{-2}}.
\end{align*}
\end{lem}

\begin{proof}
Let $n \in \llbracket 1,N \rrbracket$. Then, 
\begin{align*}
&\proba\left(\left.
\bigcup_{i \in I_{n+1}} \errgf_{i}^{(M),n}
\right|\cond^{(M),n,N}\right) \\
&= \sum_{(h_{i})_{i \in I_{n+1}} \in \vcal^{(M),n}} 
\proba\left(\left.
\forall i \in I_{n+1}, h_{i}^{(M),n} = h_{i}
\right|
\cond^{(M),n,N}
\right) \\
&\qquad \qquad  \qquad \qquad \times \proba\left(\left.
\bigcup_{i \in I_{n+1}} 
\errgf_{i}^{(M),n}
\right|
\cond^{(M),n,N} \cap \left\{
\forall i \in I_{n+1}, h_{i}^{(M),n} = h_{i}
\right\}
\right) \\
&\leq \sum_{(h_{i})_{i \in I_{n+1}} \in \vcal^{(M),n}} \sum_{i \in I_{n+1}}
\proba\left(\left.
\forall i' \in I_{n+1}, h_{i'}^{(M),n} = h_{i'}
\right|
\cond^{(M),n,N}
\right) \\
&\qquad \qquad \qquad \qquad \qquad \quad \times \proba\left(\left.
\errgf_{i}^{(M),n}
\right|
\cond^{(M),n,N} \cap \left\{
\forall i' \in I_{n+1}, h_{i'}^{(M),n} = h_{i'}
\right\}
\right).
\end{align*}
Moreover, for all $i \in I_{n+1}$ and $(h_{i})_{i \in I_{n+1}} \in \mathcal{V}^{(M),n}$, by (\ref{eqn:reecriture_event_conditional}), if we define the event $\cond^{+,(M),n,N,(h_{i})_{i \in I_{n+1}}}$ as
\begin{equation*}
\cond^{+,(M),n,N,(h_{i})_{i \in I_{n+1}}} := \cond^{(M),n,N} \cap \left\{
\forall i \in I_{n+1}, h_{i}^{(M),n} = h_{i}\right\},
\end{equation*}
then
\begin{align*}
&\proba\left(\left.
\errgf_{i}^{(M),n}
\right|
\cond^{(M),n,N} \cap \left\{
\forall i' \in I_{n+1}, h_{i'}^{(M),n} = h_{i'}
\right\}
\right) \\
&\leq \un{h_{i} \leq H} \\
&\qquad \times\proba\left(\left.\left\{
\Card\left(
E_{i,h_{i}}^{(M),n} \cap G_{i}^{(M),n}\right) = 0 \right\} \cup \left\{
\Card\left(
E_{i,h_{i}+1}^{(M),n+1}
\right) \notin W_{h_{i}+1}^{(M),\epsilon}
\right\}\right|
\cond^{+,(M),n,N,(h_{i})_{i \in I_{n+1}}}
\right).
\end{align*}
Whether $\Card\left(
E_{i,h_{i}}^{(M),n} \cap G_{i}^{(M),n}\right) = 0$ or $\Card\left(
E_{i,h_{i}+1}^{(M),n+1}
\right) \notin W_{h_{i}+1}^{(M),\epsilon}$ only depends on the number of age~$h_{i}$ seeds in patch~$i$ at the beginning of generation~$n$, and not on the past dynamics, the age of the youngest type~$1$ seeds in patch~$i$ as well as other patches, or the composition of other patches. Therefore, the event
\begin{equation*}
\left\{
\Card\left(
E_{i,h_{i}}^{(M),n} \cap G_{i}^{(M),n}\right) = 0 \right\} \cup \left\{
\Card\left(
E_{i,h_{i}+1}^{(M),n+1}
\right) \notin W_{h_{i}+1}^{(M),\epsilon}
\right\}
\end{equation*}
is independent from most of the events whose union form $\cond^{+,(M),n,N,(h_{i})_{i \in I_{n+1}}}$, and
\begin{align*}
&\proba\left(\left.
\errgf_{i}^{(M),n}
\right|
\cond^{(M),n,N} \cap \left\{
\forall i' \in I_{n+1}, h_{i'}^{(M),n} = h_{i'}
\right\}
\right) \\
&\leq \un{h_{i} \leq H} \\
&\qquad \times
\proba\left(\left.\left\{
\Card\left(
E_{i,h_{i}}^{(M),n} \cap G_{i}^{(M),n}\right) = 0 \right\} \cup \left\{
\Card\left(
E_{i,h_{i}+1}^{(M),n+1}
\right) \notin W_{h_{i}+1}^{(M),\epsilon}
\right\}\right|
\Card\left(
E_{i,h_{i}}^{(M),n}
\right) \in W_{h_{i}}^{(M),\epsilon}
\right) \\
&\leq 2 e^{-2\epsilon^{2}\lfloor gM \rfloor^{3} ((1-\lfloor gM \rfloor M^{-1})^{H}-\epsilon a_{H})^{2}M^{-2}}
\end{align*}
by Lemma \ref{lem:distr_age_of_seeds}. Therefore, 
\begin{align*}
&\proba\left(\left.
\bigcup_{i \in I_{n+1}} \errgf_{i}^{(M),n}
\right|\cond^{(M),n,N}\right) \\
&\leq 2 \left(
i_{max}^{n+1}-i_{min}^{n+1}+1
\right) e^{-2\epsilon^{2}\lfloor gM \rfloor^{3} ((1-\lfloor gM \rfloor M^{-1})^{H}-\epsilon a_{H})^{2}M^{-2}} \\
&\leq 2(i_{max}^{0}-i_{min}^{0}+1+2n+2) e^{-2\epsilon^{2}\lfloor gM \rfloor^{3} ((1-\lfloor gM \rfloor M^{-1})^{H}-\epsilon a_{H})^{2}M^{-2}} \\
\intertext{and}
&\sum_{n = 1}^{N} \proba\left(\left.
\bigcup_{i \in I_{n+1}} \errgf_{i}^{(M),n}
\right|\cond^{(M),n,N}\right) \\
&\leq \sum_{n = 1}^{N}\left(2(2n+i_{max}^{0}-i_{min}^{0}+3)\right) e^{-2\epsilon^{2}\lfloor gM \rfloor^{3} ((1-\lfloor gM \rfloor M^{-1})^{H}-\epsilon a_{H})^{2}M^{-2}} \\
&\leq \left(
2N(N+1)+2(i_{0}^{max}-i_{0}^{min})N+6N
\right)e^{-2\epsilon^{2}\lfloor gM \rfloor^{3} ((1-\lfloor gM \rfloor M^{-1})^{H}-\epsilon a_{H})^{2}M^{-2}},
\end{align*}
which allows us to conclude. 
\end{proof}

Obtaining a similar result for the quantity in (\ref{eqn:terme_moche_2}) does not require conditioning, since $h^{(M),0}$ is determined by the initial condition. 

\begin{lem}\label{lem:terme_moche_2}
For all $M \geq 2$ such that  
$\left(1-\lgmr M^{-1}\right)^{H}-\lgmr^{-1} > 0$, 
\begin{align*}
&\proba\left(\left.
\bigcup_{i \in I_{1}} \errgf_{i}^{(M),0}
\right|
\left(
\overline{\errpar}_{i}^{(M),n'} \right)_{\substack{0 \leq n' \leq N-1 \\ i \in I_{n'+1}}}\right) \leq 2\left(i_{max}^{0}-i_{min}^{0}+3\right)
e^{-2\epsilon^{2} \lfloor gM \rfloor^{3} ((1-\lfloor gM \rfloor M^{-1})^{H}-\epsilon a_{H})^{2} M^{-2}}.
\end{align*}
\end{lem}

\begin{proof}
By \ref{eqn:reecriture_event_conditional} and by Lemma~\ref{lem:distr_age_of_seeds}, given the initial condition, 
\begin{align*}
&\proba\left(\left.
\bigcup_{i \in I_{1}} \errgf_{i}^{(M),0}
\right|
\left(
\overline{\errpar}_{i}^{(M),n'} \right)_{\substack{0 \leq n' \leq N-1 \\ i \in I_{n'+1}}}\right) \\
&\leq \sum_{i \in I_{1}} \un{h_{i}^{(M),0} \leq H} \\
&\qquad \, \times \proba\left(\left\{
\Card\left(
E_{i,h_{i}^{(M),0}}^{(M),0}\cap G_{i}^{(M),0}
\right)  = 0
\right\} \cup \left\{
\Card\left(E_{i,h_{i}^{(M),0}+1}^{(M),1}\right) \notin W_{h_{i}^{(M),0}+1}^{(M),\epsilon}
\right\}\left|
\left(\overline{\mathrm{Par}}_{i}^{(M),n}\right)_{\substack{n' \leq N-1 \\ i \in I_{n'+1}}}
\right.\right) \\
&\leq \sum_{i \in I_{1}} \proba\left(\left\{
\Card\left(
E_{i,h_{i}^{(M),0}}^{(M),0}\cap G_{i}^{(M),0}
\right)  = 0
\right\} \cup \left\{
\Card\left(E_{i,h_{i}^{(M),0}+1}^{(M),1}\right) \notin W_{h_{i}^{(M),0}+1}^{(M),\epsilon}
\right\}\right) \\
&\leq 2\left(
i_{max}^{0}-i_{min}^{0} +3
\right)
e^{-2\epsilon^{2}\lfloor gM \rfloor^{3} ((1-\lfloor gM \rfloor M^{-1})^{H}-\epsilon a_{H})^{2}M^{-2}}
\end{align*}
given the initial condition. 
\end{proof}

We can now show Theorem \ref{thm:cvg_spom} for our specific initial condition. 

\begin{proof}(Theorem \ref{thm:cvg_spom}, post-transition period)
Let $M \geq 2$ be such that  
$\left(1-\lgmr M^{-1}\right)^{H}-\lgmr^{-1} > 0$. The result is clear for $N = 0$. For $N \in \nmath \backslash \{0\}$, by definition of the event $\dboa_{i}^{(M),n}$,
\begin{align*}
&\proba\left(\bigcap_{n = 0}^{N} \left(
\left\{
\forall i \in \zmath, O_{i}^{(M),n} = O_{i}^{(M),\infty,n}
\right\} \cap \left\{
\forall i \in \zmath, h_{i}^{(M),n} = h_{i}^{(M),\infty,n}
\right\}
\right)\right) \\
&= \proba\left(
\bigcap_{n = 0}^{N}\bigcap_{i \in \zmath} \overline{\dboa}_{i}^{(M),n}
\right)\\
&= 1 - \proba\left(
\bigcup_{n = 0}^{N}\bigcup_{i \in \zmath} \dboa_{i}^{(M),n}
\right) \\
&= 1 - \proba\left(
\bigcup_{n = 1}^{N}\bigcup_{i \in I_{n}} \dboa_{i}^{(M),n}
\right)
\end{align*}
by Eq.~(\ref{eqn:equality_theorem_2}). Therefore, by Eq.~(\ref{eqn:equality_theorem_1}), 
\begin{align*}
&\proba\left(\bigcap_{n = 0}^{N} \left(
\left\{
\forall i \in \zmath, O_{i}^{(M),n} = O_{i}^{(M),\infty,n}
\right\} \cap \left\{
\forall i \in \zmath, h_{i}^{(M),n} = h_{i}^{(M),\infty,n}
\right\}
\right)\right) \\
&\geq 1 - \proba\left(\bigcup_{n = 0}^{N-1} \bigcup_{i \in I_{n+1}} 
\errpar_{i}^{(M),n} \cup \errg_{i}^{(M),n} \cup \errf_{i}^{(M),n+1}
\right),
\end{align*}
which can be bounded from below by $1-$ Eq.(\ref{eqn:terme_moche_1}) $-$ Eq.(\ref{eqn:terme_moche_2}) $-$ Eq.(\ref{eqn:terme_moche_3}). By Lemmas~\ref{lem:maj_p_rm}, \ref{lem:terme_moche_2} and \ref{lem:terme_moche_3}, we can show that each of the three terms converges to $0$ when $M \to + \infty$, allowing us to conclude. 
\end{proof}

We now explain how to generalize the proof of Theorem~\ref{thm:cvg_spom} to a more general sequence of initial conditions not necessarily satisfying (IC1) and (IC2). Since $(\xi^{(M),n},h^{(M),n})_{n \in \nmath}$ is Markovian for all $M \geq 2$, it is sufficient to show that the process does not deviate from the BOA process during the transition period, and that at the end of this period, $(\xi^{(M),H+1},h^{(M),H+1})$ satisfies (IC1) and (IC2). 

In order to do so, let $n \in \llbracket 0,H \rrbracket$ be a generation from the transition period, and let $i \in \zmath$. We distinguish three cases. 
\begin{enumerate}
\item If $O_{i}^{\infty} = 0$, then by condition~(A) patch $i$ is initially empty, so all viable $h_{i}^{\infty,0}$ seeds it contains are of type~$0$. We are then in the same situation as during the post-transition period. 
\item If $O_{i}^{\infty} = 1$ and $h_{i}^{(M),n} > n-1$, then the age $h_{i}^{(M),n}$ seeds in patch $i$ were already present initially. Similarly as before and using condition~(C), we can show that with high probability, the number of remaining type~$1$ seeds of age $h_{i}^{(M),n}$ is roughly equal to $g_{i}(1-g)^{n}M$, and at least one of them germinates during generation~$n$. 
\item If $O_{i}^{\infty} = 1$ and $h_{i}^{(M),n} \leq n+1$, then the age $h_{i}^{(M),n}$ seeds in patch $i$ were not present initially. 
By Lemma~\ref{lem:maj_p_rm}, they are all of the same type with high probability, and we are back to the case considered during the post-transition period. 
\end{enumerate}

\section{Extinction threshold for the $k$-parent WFSB metapopulation process}\label{sec:ext_threshold}
This section is devoted to the proof of Theorem \ref{thm:proba_critique}, that is, to the proof of the existence of a critical extinction probability $p_{crit}(H)$ depending only on the maximal dormancy duration $H$. In order to do so, we will first formalize the coupling between the $k$-parent WFSB metapopulation process and a BOA process. Then, we will explain how the issue of occupied patches in the BOA process can be seen as a percolation problem. We will conclude using a specific case of Eq.(4) in \cite{hartarsky2021generalised}.

\subsection{Coupling between the $k$-parent WFSB metapopulation process and the BOA process}\label{sec:BOA_process}
In all that follows, let $(\xi, h) \in \fcal_{M} \times \hcal_{M}$, let $(\xi^{n},h^{n})_{n \in \nmath}$ be the $k$-parent WFSB metapopulation process with parameters $(M,H,g,c,p)$ and initial condition $(\xi,h)$, and let $(O^{k,n},h^{k,n})_{n \in \nmath}$ be the associated $k$-parent occupancy process. In order to couple a BOA process to $(\xi^{n},h^{n})_{n \in \nmath}$, for all $n \in \nmath^{*}$, we denote by $(\textrm{Ext}_{i}^{n})_{i \in \zmath}$ the extinction events used to define $(\xi^{n},h^{n})$ given $(\xi^{n-1},h^{n-1})$. In other words, for all $n \in \nmath^{*}$ and $i \in \zmath$, $\textrm{Ext}_{i}^{n} = 1$ if, and only if the patch $i$ was extinct during the $n$-th generation. We then define the coupled BOA process $(O^{\infty,n},h^{\infty,n})_{n \in \nmath}$ as the BOA process with parameters $(H,p)$ and initial condition $(O^{k,0},h^{k,0})$, constructed using the extinction events $(\textrm{Ext}_{i}^{n})_{i \in \zmath, n \in \nmath^{*}}$: for all $n \in \nmath$, $(O^{\infty,n+1},h^{\infty,n+1})$ is constructed using $(O^{\infty,n},h^{\infty,n})$ and the extinction events $(\textrm{Ext}_{i}^{n+1})_{i \in \zmath}$. This coupling satisfies the following property, whose proof is postponed until later in this section for the sake of clarity.

\begin{prop}\label{prop:coupling_BOA}
For all $n \in \nmath$ and $i \in \zmath$, 
\begin{equation*}
O_{i}^{k,n} \leq O_{i}^{\infty,n} \text{  and  } h_{i}^{k,n} \geq h_{i}^{\infty,n}.
\end{equation*}
\end{prop}

Therefore, at any generation $n \in \nmath$, the set of patches which contain nonexpired seeds in the $k$-parent WFSB metapopulation process is included in the set of reachable patches in the BOA process (that is, patches $i$ such that $O_{i}^{\infty,n} \un{h_{i}^{\infty,n} \leq H} = 1$). In particular, a consequence of this coupling is the following corollary.

\begin{cor}\label{corr:coupling_BOA}
For all $n \in \nmath$,
\begin{equation*}
\proba\left(
1 - \prod_{(i,j) \in \zmath \times \llbracket 1,M \rrbracket} \left(1 - \mathds{1}_{\{h_{i,j}^{n} \leq H \}} \xi_{i,j}^{n}\right) = 1
\right) \leq \proba\left(
1 - \prod_{i \in \zmath} \left(
1 - \mathds{1}_{\{h_{i}^{\infty,n} \leq H\}} O_{i}^{\infty,n}
\right) = 1
\right).
\end{equation*}
\end{cor}

\begin{proof}
Let $n \in \nmath$. By definition of the $k$-parent occupancy process, for all $(i,j) \in \zmath \times \llbracket 1,M \rrbracket$, 
\begin{equation*}
\xi_{i,j}^{n} \leq O_{i}^{k,n}.
\end{equation*}
Indeed, both $\xi_{i,j}^{n}$ and $O_{i}^{k,n}$ are $\{0,1\}$-valued, and if $\xi_{i,j}^{n} = 1$, then $O_{i}^{k,n} = 1$. 

Moreover, if $O_{i}^{k,n} = 1$, then $h_{i}^{k,n}$ is the age of the youngest type $1$ seed in patch $i$. 
Therefore, for all $(i,j) \in \zmath \times \llbracket 1,M \rrbracket$, if $\xi_{i,j}^{n} = 1$, then $h_{i,j}^{n} \geq h_{i}^{k,n}$.
We deduce that 
\begin{equation*}
\mathds{1}_{\{h_{i,j}^{n} \leq H\}}\xi_{i,j}^{n}
\leq \mathds{1}_{\{h_{i,j}^{k,n} \leq H \}} O_{i}^{k,n}.
\end{equation*}
By Proposition \ref{prop:coupling_BOA}, we obtain
\begin{equation*}
\mathds{1}_{\{h_{i,j}^{n} \leq H\}} \xi_{i,j}^{n} \leq \mathds{1}_{\{h_{i}^{\infty,n} \leq H\}} O_{i}^{\infty,n}.
\end{equation*}
Taking the product over all $(i,j) \in \zmath \times \llbracket 1,M \rrbracket$ yields
\begin{align*}
1 - \prod_{(i,j) \in \zmath \times \llbracket 1,M \rrbracket} \left(
1 - \mathds{1}_{\{h_{i,j}^{n} \leq H\}} \xi_{i,j}^{n}
\right) \leq & 
1 - \prod_{(i,j) \in \zmath \times \llbracket 1,M \rrbracket}\left(
1 - \mathds{1}_{\{h_{i}^{\infty,n} \leq H\}} O_{i}^{\infty,n}
\right) \\
\leq &
1 - \prod_{i \in \zmath}\left(
1 - \mathds{1}_{\{h_{i}^{\infty,n} \leq H\}} O_{i}^{\infty,n}
\right).
\end{align*}
since all the terms of the product are $\{0,1\}$-valued, and we can conclude.
\end{proof}

We now show Proposition \ref{prop:coupling_BOA}. 

\begin{proof}(Proposition \ref{prop:coupling_BOA}) 
We show the result by induction. 
For $n = 0$, since $(O^{\infty,0},h^{\infty,0}) = (O^{k,0},h^{k,0})$, 
\begin{align*}
1 - \prod_{(i,j) \in \zmath \times \llbracket 1,M \rrbracket} \left(
1 - \un{h_{i,j}^{0} \leq H} \xi_{i,j}^{0}
\right) &= 
1 - \prod_{i \in \zmath} \prod_{j \in \llbracket 1,M \rrbracket} \left(
1 - \un{h_{i,j}^{0} \leq H} \xi_{i,j}^{0}
\right) \\
&= 1 - \prod_{i \in \zmath} \left(
1 - \un{h_{i}^{k,0} \leq H} O_{i}^{k,0}
\right) \\
&= 1 - \prod_{i \in \zmath} \left(
1 - \un{h_{i}^{\infty,0} \leq H} O_{i}^{\infty,n}
\right),
\end{align*}
so the result is true for $n = 0$.

Let then $n \in \nmath$, and we assume that for all $i \in \zmath$, 
\begin{equation*}
O_{i}^{k,n} \leq O_{i}^{\infty,n} \text{  and  } h_{i}^{k,n} \geq h_{i}^{\infty,n}.
\end{equation*}
Let $i \in \zmath$. We first show that $O_{i}^{k,n+1} \leq O_{i}^{\infty,n+1}$. Since $O_{i}^{k,n+1} \in \{0,1\}$, if $O_{i}^{\infty,n+1} = 1$, then $O_{i}^{k,n+1} \leq O_{i}^{\infty,n+1}$. Therefore, we assume $O_{i}^{\infty,n+1} = 0$. 
Notice that by definition of the BOA process, $(O_{i}^{\infty,n})_{n \in \nmath}$ is an increasing sequence. Indeed, for all $n \geq 0$ and $i \in \zmath$, $O_{i}^{\infty,n+1}$ is set equal to $O_{i}^{\infty,n}$ or $1$, so if $O_{i}^{\infty,n} = 1$, then $O_{i}^{\infty,n+1} \in \{O_{i}^{\infty,n},1\} = \{1\}$, and for all $n' \geq n$,  we have $O_{i}^{\infty,n'} = 1$.
This means that $O_{i}^{\infty,n+1} = 0$ implies $O_{i}^{\infty,n} = 0$ and $O_{i}^{k,n} = 0$. Moreover, it also means that both neighbouring patches were either extinct or not reachable in generation $n$. We deduce
\begin{align*}
\left(1 - \textrm{Ext}_{i+1}^{n+1}\right) O_{i+1}^{\infty,n} \mathds{1}_{\{h_{i+1}^{\infty,n} \leq H\}} &= 0 \\
\text{and } \left(1 - \textrm{Ext}_{i-1}^{n+1}\right) O_{i-1}^{\infty,n} \mathds{1}_{\{h_{i-1}^{\infty,n} \leq H\}} &= 0.
\end{align*}
Therefore, by the induction hypothesis,
\begin{align*}
\left(1 - \textrm{Ext}_{i+1}^{n+1}\right) O_{i+1}^{k,n} \mathds{1}_{\{h_{i+1}^{k,n} \leq H\}} &= 0 \\
\text{and } \left(1 - \textrm{Ext}_{i-1}^{n+1}\right) O_{i-1}^{k,n} \mathds{1}_{\{h_{i-1}^{k,n} \leq H\}} &= 0,
\end{align*}
which means that the patches $i-1$ and $i+1$ are either extinct or containing only ghost type $0$ seeds. Combined with the knowledge that $O_{i}^{k,n} = 0$,  we obtain that $O_{i}^{k,n+1} = 0$. 

We now have to show that $h_{i}^{k,n+1} \geq h_{i}^{\infty,n+1}$. Since $h_{i}^{k,n} \geq h_{i}^{\infty,n}$ and since $h_{i}^{k,n+1}$ (resp. $h_{i}^{\infty,n+1}$) is either equal to $h_{i}^{k,n} + 1$ (resp. $h_{i}^{\infty,n} + 1$) or equal to $0$, the only potential issue is when $h_{i}^{k,n+1} = 0$. Let us assume that $h_{i}^{k,n+1} = 0$. This means that new seeds were just produced, and implies that 
\begin{equation*}
1 - \prod_{i' = i-1}^{i+1}\left(
1 - \left(1 - \textrm{Ext}_{i'}^{n+1}\right) O_{i'}^{k,n} \mathds{1}_{\{h_{i'}^{k,n} \leq H\}}
\right) = 1,
\end{equation*}
i.e, that non-expired seeds were present in at least one of the patches $\{i-1$, $i$, $i+1\}$, and that at least one of these patches was not affected by an extinction event. Moreover,  if
\begin{align*}
 \prod_{i' = i-1}^{i+1}\left(
1 - \left(1 - \textrm{Ext}_{i'}^{n+1}\right) O_{i'}^{\infty,n} \mathds{1}_{\{h_{i'}^{\infty,n} \leq H\}}
\right) = 0, \\
\intertext{then $h_{i}^{\infty,n+1} = 0$. Using the induction hypothesis yields}
 \prod_{i' = i-1}^{i+1}\left(
1 - \left(1 - \textrm{Ext}_{i'}^{n+1}\right) O_{i'}^{\infty,n}  \mathds{1}_{\{h_{i'}^{\infty,n} \leq H\}}
\right) \leq & \prod_{i' = i-1}^{i+1}\left(
1 - \left(1 - \textrm{Ext}_{i'}^{n+1}\right)  O_{i'}^{k,n}  \mathds{1}_{\{h_{i'}^{k,n} \leq H\}}
\right)\\
= & \, 0,
\end{align*}
hence $h_{i}^{\infty,n+1} = 0 = h_{i}^{k,n+1}$ and we can conclude.
\end{proof}

\subsection{Percolation problem}
In order to show Theorem \ref{thm:proba_critique}, we now link the BOA process to a percolation problem. More specifically, we rephrase the question of which patches are reachable in the BOA process as an oriented site percolation problem. Indeed, we can see patch $i \in \zmath$ in generation $n \in \nmath$ as the site $(i,n)$ of the space $\zmath \times \nmath$. Each site $(i,n) \in \zmath \times \nmath$ is \textit{open} (the analog of \textit{non-extinct} in the terminology of percolation) with probability $1-p$, and \textit{closed} (i.e, extinct) otherwise. Reachable patches can be seen as sites of the space $\zmath \times \nmath$ linked to a site of $\zmath \times \{0\}$ by a path of open sites
\begin{equation*}
(i_{0},n_{0}) = (i_{0},0) \longrightarrow (i_{1},n_{1}) \longrightarrow ...
\longrightarrow (i_{L},n_{L}) = (i,n)
\end{equation*}
such that $O_{i_{0}}^{\infty,0}\times \un{h_{i_{0}}^{\infty,n} \leq H} = 1$, $i_{1} \in \{i_{0}-1,i_{0},i_{0}+1\}$, $n_{1}-n_{0} \in \llbracket 1,H-h_{i_{0}}^{\infty,n} + 1 \rrbracket$, and for all $l \in \llbracket 2,L \rrbracket$, 
\begin{equation}\label{cond:path}
i_{l} \in \{i_{l-1}-1, i_{l-1}, i_{l-1} + 1\} \qquad \text{ and } \qquad n_{l} - n_{l-1} \in \llbracket 1,H+1 \rrbracket. 
\end{equation}

For all $n \in \nmath$, let $S_{n}(p)$ be the set of all the sites $(i,n)$ with $i \in \zmath$ that are connected to $(0,0)$ by a path of open sites satisfying $i_{1} \in \{i_{0}-1,i_{0},i_{0}+1\}$, $n_{1}-n_{0} \in \llbracket 1,H+ 1 \rrbracket$ and (\ref{cond:path}). Equivalently, let $(O^{\{0\},n},h^{\{0\},n})_{n \in \nmath}$ be the BOA process with parameters $(H,p)$ and initial condition satisfying:
\begin{enumerate}
\item $O_{0}^{\{0\},0} = 1 \text{ and } h_{0}^{\{0\},0} = 0$.
\item For all $i \in \zmath \backslash \{0\}$, $O_{i}^{\{0\},0} = 0$ and $h_{i}^{\{0\},0} = 0$.
\end{enumerate}
We can then define $S_{n}(p)$ as
\begin{equation*}
S_{n}(p) := \left\{
i \in \zmath : O_{i}^{\{0\},n} \un{h_{i}^{\{0\},n} \leq H} = 1
\right\}.
\end{equation*}
Under this notation, a direct consequence of Eq. (4) in \cite{hartarsky2021generalised} is the following proposition.
\begin{prop}\label{prop:percolation}
There exists a unique $p_{crit}(H) \in (0,1)$ such that
\begin{align*}
\forall p \in [0, p_{crit}(H)),\proba\left(
\forall n \in \nmath, S_{n}(p) \neq \emptyset
\right) &> 0 \\
\text{and }\quad \forall p \in (p_{crit}(H),1], \proba\left(
\forall n \in \nmath, S_{n}(p) \neq \emptyset
\right) &= 0.
\end{align*}
\end{prop}
What remains to show is that $p_{crit}(H)$ is indeed the extinction threshold we are looking for.

\subsection{Proof of Theorem \ref{thm:proba_critique}}
In order to prove Theorem \ref{thm:proba_critique}, we make three observations. First, for all $n \in \nmath$, the event $\{S_{n}(p) \neq \emptyset\}$ is the same as the event
\begin{equation*}
\left\{
1 - \prod_{i \in \zmath} \left(
1 - O_{i}^{\{0\},n} \un{h_{i}^{\{0\},n} \leq H}
\right) = 1
\right\}.
\end{equation*}
%Indeed, if there exists $i_{0} \in \zmath$ such that $O_{i_{0}}^{\{0\},n} \times \un{h_{i_{0}}^{\{0\},n} \leq H} = 1$, then $S_{n} \neq \emptyset$ and 
%\begin{equation*}
%\prod_{i \in \zmath}
%\left(
%1 - O_{i}^{\{0\},n} \times \un{h_{i}^{\{0\},n} \leq H}
%\right)
%= \left(
%1 - O_{i_{0}}^{\{0\},n} \times \un{h_{i_{0}}^{\{0\},n} \leq H}
%\right) 
%\times \prod_{i \in \zmath \backslash \{i_{0}\}} 
%\left(
%1 - O_{i}^{\{0\},n} \times \un{h_{i}^{\{0\},n} \leq H}
%\right)
%= 0.
%\end{equation*}
Moreover, for all finite subsets $\lcal$ of $\zmath$, let $(O^{\lcal},h^{\lcal}) \in \fcal^{\infty} \times \hcal^{\infty}$ satisfy the two following conditions:
\begin{itemize}
\item For all $i \in \lcal$, $O_{i}^{\lcal} = 1$ and $h_{i}^{\lcal} = 0$.
\item For all $i \in \zmath \backslash \lcal$, $O_{i}^{\lcal} = 0$ and $h_{i}^{\lcal} = 0$.
\end{itemize}
Let also $(O^{\lcal,n},h^{\lcal,n})_{n \in \nmath}$ be the BOA process with parameters $(H,p)$ and initial condition $(O^{\lcal},h^{\lcal})$. That is, $(O^{\lcal,n},h^{\lcal,n})_{n \in \nmath}$ is the BOA process starting from the state where all the patches in $\lcal$ are of type $1$ and all the patches in $\lcal^{c}$ of type $0$. 
Notice that if $\lcal = \{0\}$, then the definition of $\left(O^{\{0\},n},h^{\{0\},n}\right))_{n \in \nmath}$ matches the one given above.
We then have the following result.

\begin{lem}\label{lem:percolation_1}
For all finite subset $\lcal$ of $\zmath$ and for all $n \in \nmath$, 
\begin{equation*}
\proba\left(
1 - \prod_{i \in \zmath} \left(
1 - O_{i}^{\lcal,n} \un{h_{i}^{\lcal,n} \leq H}
\right) = 0
\right)
\geq
\proba\left(
1 - \prod_{i \in \zmath} 
\left(
1 - O_{i}^{\{0\},n} \un{h_{i}^{\{0\},n} \leq H}
\right) = 0
\right)^{\Card(\lcal)}.
\end{equation*}
\end{lem}
This lemma gives a lower bound of the probability that no patches are reachable in at least $n$ generations in the BOA process starting from the patches in $\lcal$, each one of them containing type $1$ seeds of age $0$. This lower bound involves the probability that no patches are reachable in at least $n$ generations starting from \textit{only one patch}, which is used in the definition of $p_{crit}(H)$.

\begin{proof}
Let $n \in \nmath$ and let $\lcal$ be a finite subset of $\zmath$. First, we observe that if we couple all the BOA processes considered by constructing them using the same extinction events, 
\begin{equation*}
1 - \prod_{i \in \zmath} \left(
1 - O_{i}^{\lcal,n} \un{h_{i}^{\lcal,n} \leq H}
\right)
= 1 - \prod_{i' \in \lcal} \left[\prod_{i \in \zmath} \left(
1 - O_{i}^{\{i'\},n} \un{h_{i}^{\{i'\},n} \leq H}
\right)\right].
\end{equation*}
Indeed, each one of the reachable patches in the BOA process with initial conditions $(O^{\lcal},h^{\lcal})$ is connected by a path of nonextinct patches to a patch in $\lcal$, and so there exists $i_{0} \in \lcal$ such as the patch is also reachable in the BOA process with initial condition $(O^{\{i_{0}\}},h^{\{i_{0}\}})$. We can then use the fact that all the quantities appearing in the product are $\{0,1\}$-valued. 

Moreover, for $i_{0}, i_{1} \in \lcal$ and again using our coupling, knowing that no patch is reachable in $n$ generations starting from $i_{0}$ increases the probability that no patch is reachable in $n$ generations starting from $i_{1}$. Indeed, informally, the fact that no patch is reachable starting from $i_{0}$ "blocks" some patches, which cannot be used by a path linking $i_{1}$ to other patches. Therefore, 
\begin{align*}
&\proba\left(\left.
\prod_{i \in \zmath} \left(
1 - O_{i}^{\{i_{1}\},n} \un{h_{i}^{\{i_{1}\},n} \leq H}
\right) = 1
\right|
\prod_{i \in \zmath} \left(
1 - O_{i}^{\{i_{0}\},n} \un{h_{i}^{\{i_{0}\},n} \leq H}
\right) = 1
\right) \\
&\geq \proba \left(
\prod_{i \in \zmath} \left(
1 - O_{i}^{\{i_{1}\},n} \un{h_{i}^{\{i_{1}\},n} \leq H}
\right) = 1
\right),
\end{align*}
and hence for $i_{0} \in \lcal$,
\begin{align*}
\proba\left(
1-\prod_{i \in \zmath} \left(
1 - O_{i}^{\lcal,n} \un{h_{i}^{\lcal,n} \leq H}
\right) = 0
\right) 
&= \proba\left(
1-\prod_{i' \in \lcal}\left[\prod_{i \in \zmath} \left(
1 - O_{i}^{\{i'\},n} \un{h_{i}^{\{i'\},n} \leq H}\right]
\right) = 0
\right) \\
&= \proba\left(
\bigcap_{i' \in \lcal} \left\{
\prod_{i \in \zmath} \left(
1 - O_{i}^{\{i'\},n} \un{h_{i}^{\{i'\},n} \leq H}
\right) = 1
\right\}
\right) \\
&\geq \proba\left(
\prod_{i \in \zmath} \left(
1 - O_{i}^{\{i_{0}\},n} \un{h_{i}^{\{i_{0}\},n} \leq H}
\right) = 1
\right)^{\Card(\lcal)} \\
&= \proba\left(
\prod_{i \in \zmath} \left(
1 - O_{i}^{\{0\},n} \un{h_{i}^{\{0\},n} \leq H}
\right) = 1
\right)^{\Card(\lcal)},
\end{align*}
where the invariance by translation of the process is used to pass from the last but first to the last line.
\end{proof}

We recall that the $k$-parent occupancy process associated to $(\xi^{n},h^{n})_{n \in \nmath}$ is denoted by $(O^{k,n},h^{k,n})_{n \in \nmath}$. The coupling based on the extinction events also yields the following lemma.
\begin{lem}\label{lem:percolation_2}
Let $\lcal \subset \zmath$ be the set defined as
\begin{equation*}
\lcal := \left\{i \in \zmath : O_{i}^{k,0} = 1\right\}.
\end{equation*}
Then, 
\begin{equation*}
\proba\left(
1 - \prod_{i \in \zmath} \left(
1 - O_{i}^{k,n} \un{h_{i}^{k,n} \leq H}
\right) = 0
\right)
\geq 
\proba\left(
1 - \prod_{i \in \zmath} \left(
1 - O_{i}^{\lcal,n} \un{h_{i}^{\lcal,n} \leq H}
\right) = 0
\right).
\end{equation*}
\end{lem}
Indeed, if $(\xi^{n},h^{n})_{n \in \nmath}$ (hence $(O^{k,n},h^{k,n})_{n \in \nmath}$) and $(O^{\lcal,n},h^{\lcal,n})_{n \in \nmath}$ are constructed using the same extinction events, then all the patches occupied by the $k$-parent WFSB metapopulation process are also reachable by the BOA process $(O^{\lcal,n},h^{\lcal,n})_{n \in \nmath}$. Here deviations from the BOA process $(O^{\lcal,n},h^{\lcal,n})_{n \in \nmath}$ can also occur if the youngest type $1$ seeds in $(\xi^{0},h^{0})$ are \textit{not} of age $0$, but older.

We can now prove Theorem \ref{thm:proba_critique}.

\begin{proof}
(Theorem \ref{thm:proba_critique}) 
Let $p_{crit}(H)$ be given by Proposition \ref{prop:percolation}. We assume that $p > p_{crit}(H)$. Let also $n \in \nmath$, and let $\lcal \subset \zmath$ be defined as in Lemma \ref{lem:percolation_2}.

By Lemma \ref{lem:percolation_2},
\begin{align*}
\proba\left(
\forall i \in \zmath, O_{i}^{k,n} \un{h_{i}^{k,n} \leq H} = 0
\right)
&= \proba\left(
1 - \prod_{i \in \zmath} \left(
1 - O_{i}^{k,n} \un{h_{i}^{k,n} \leq H}
\right) = 0
\right) \\
&\geq \proba\left(
1 - \prod_{i \in \zmath}\left(
1 - O_{i}^{\lcal,n} \un{h_{i}^{\lcal,n} \leq H}
\right) = 0
\right).
\end{align*}
Using Lemma \ref{lem:percolation_1}, we obtain
\begin{align*}
\proba\left(
\forall i \in \zmath, O_{i}^{k,n} \un{h_{i}^{k,n} \leq H} = 0
\right)
&\geq \proba\left(
1 - \prod_{i \in \zmath}\left(
1 - O_{i}^{\{0\},n} \un{h_{i}^{\{0\},n} \leq H}
\right) = 0
\right)^{\Card(\lcal)} \\
&= \proba\left(
S_{n}(p) = \emptyset
\right)^{\Card(\lcal)}.
\end{align*}
Therefore, 
\begin{align*}
\lim\limits_{n \to + \infty} \proba\left(
\forall i \in \zmath, O_{i}^{k,n} \un{h_{i}^{k,n} \leq H} = 0
\right) &\geq \lim\limits_{n \to + \infty} \proba(S_{n}(p) =  \emptyset)^{\Card(\lcal)} \\
&\geq 1
\end{align*}
by Proposition \ref{prop:percolation}, and we can conclude.
\end{proof}

\section{Appendix - Computation of $p_{crit}(H)$}
In this section, we briefly explain how to compute $p_{crit}(H)$, and how to implement this approach and obtain an approximation for $p_{crit}(H)$. The computation method is a direct adaptation of Section~$3$ in \cite{durrett1984oriented}. Our goal here is not to obtain very precise approximations, but rather to have a rough estimate of $p_{crit}(H)$, and use it to assess the impact of the presence of a seed bank on the extinction threshold. 

We first introduce the following notation. For all $i \in \zmath$ and $n \in \nmath$, let $U^{i,n}$ be a random variable such that $U^{i,n} \sim \textrm{Unif}([0,1])$. We assume that all the random variables $(U^{i,n})_{i \in \zmath, n \in \nmath}$ are independent. For all $p \in [0,1]$, let $\mathcal{S}_{p}$ be the set defined as 
\begin{equation*}
\mathcal{S}_{p} := \left\{
(i,n) : i \in \zmath, n \in \nmath \text{ and } U^{i,n} \geq p
\right\}.
\end{equation*}
$\mathcal{S}_{p}$ can be interpreted as the set of patches which would be non-extinct, if the extinction probability was equal to $p$. 

For all $x,y \in \zmath$, $n^{(x)}, n^{(y)} \in \nmath$, $H \in \nmath$ and $p \in [0,1]$, we will say that $(x,n^{(x)})$ is \textit{(H,p)-reachable from} $(y,n^{(y)})$, and denote it as $(y,n^{(y)}) \xrightarrow[(H,p)]{} (x,n^{(x)})$, if there exists $L \in \nmath$, $x_{0}$, $x_{1}$,..., $x_{L} \in \zmath$ and $n_{0}$, $n_{1}$,..., $n_{L} \in \nmath$ such that:
\begin{enumerate}
\item $x_{0} = y$, $n_{0} = n^{(y)}$, $x_{l} = x$ and $n_{l} = n^{(x)}$,
\item $\forall l \in \llbracket 1,L \rrbracket$, $x_{l} \in \{x_{l-1} - 1, x_{l-1}, x_{l-1} + 1\}$ and $1 \leq n_{l} - n_{l-1}  \leq H+1$,
\item $\forall l \in \llbracket 1,L \rrbracket$, $(x_{l},n_{l}) \in \mathcal{S}_{p}$.
\end{enumerate}
In other words, $(y,n^{(y)}) \xrightarrow[(H,p)]{} (x,n^{(x)})$ if there exists a path of open sites going from $(y,n^{(y)})$ to $(x,n^{(x)})$, spending at most $H$ generations in each patch.

Moreover, for all $p \in [0,1]$ and $n \in \nmath$, let $\mathbf{\bar{\xi}_{n}}(H,p)$ be the set defined as:
\begin{equation*}
\mathbf{\bar{\xi}_{n}}(H,p) := \left\{
x \in \zmath : \exists h_{x}, h_{y} \in \llbracket 0,H \rrbracket, \exists y \in \zmath \backslash(\nmath \backslash \{0\}), (y,h_{y}) \xrightarrow[(H,p)]{} (x,n+h_{x})
\right\},
\end{equation*}
and let $\bar{r}_{n}(H,p) := \sup\mathbf{\bar{\xi}_{n}}(H,p)$. $\mathbf{\bar{\xi}_{n}}(H,p)$ is akin to the set of patches which are reachable in $n$ generations in a BOA process with parameters $(H,p)$, but starting from an infinite number of patches. 

A direct adaptation of Section~$3$ from \cite{durrett1984oriented} yields the following result.

\begin{lem}
For all $H \in \nmath$, 
\begin{equation*}
p_{crit}(H) := \max \left\{
p \in [0,1] : \lim\limits_{n \to + \infty} \frac{\bar{r}_{n}(H,p)}{n} \geq 0
\right\}.
\end{equation*}
\end{lem}

Therefore, in order to compute $p_{crit}(H)$, it is possible to simulate the random variable $\bar{r}_{n}(H,p)$ for a large value of $n$ and for different values of $p$. 

Let $H \in \nmath$. In order to obtain an approximation for $p_{crit}(H)$, we first define some approximations for $\mathbf{\bar{\xi}_{n}}(H,p)$ and $\bar{r}_{n}(H,p)$. Let $p \in [0,1]$. For all $x,y \in \llbracket -10500,10500 \rrbracket$ and $n^{(x)}, n^{(y)} \in \llbracket 0,10000 \rrbracket$, we will say that $(x,n^{(x)})$ 
is \textit{approximatively $(H,p)$-reachable from} $(y,n^{(y)})$, and denote it as 
\begin{equation*}
(y,n^{(y)}) \xrightarrow[\mathrm{Approx}(H,p)]{} (x,n^{(x)}), 
\end{equation*}
if there exists $L \in \nmath$, $x_{0}$,... , $x_{L} \in \llbracket -10500,10500 \rrbracket$ and $n_{0}$,... ,$n_{L} \in \llbracket 0,10000 \rrbracket$ such that:
\begin{enumerate}
\item $x_{0} = y$, $n_{0} = n^{(y)}$, $x_{l} = x$ and $n_{l} = n^{(x)}$.
\item $\forall l \in \llbracket 1,L \rrbracket$, $x_{l} \in \{x_{l-1} - 1, x_{l-1}, x_{l-1} + 1\}$ and $1 \leq n_{l} - n_{l-1}  \leq H+1$.
\item $\forall l \in \llbracket 1,L \rrbracket$, if $x_{l} \neq -10500$, then $(x_{l},n_{l}) \in \mathcal{S}_{p}$ and $x_{l} \neq 10500$. 
\end{enumerate}
Therefore, in the approximation, the paths linking two sites together have to remain in the domain $\llbracket -10500, 10500 \rrbracket$, with extra conditions at the border of the domain. Since the value of the quantity we are interested in depends on the presence of paths staying close to the centre of the domain, we can assume that the border conditions chosen will not affect the approximate value. 

We then define
\begin{equation*}
\mathbf{\textrm{Approx}(\bar{\xi}_{n}}(H,p)) := \left\{
x \in \zmath : \exists h_{x}, h_{y} \in \llbracket 0,H \rrbracket, \exists y \in \zmath \backslash(\nmath \backslash \{0\}), (y,h_{y}) \xrightarrow[\mathrm{Approx}(H,p)]{} (x,n+h_{x})
\right\},
\end{equation*}
and let $\mathrm{Approx}(\bar{r}_{n}(H,p)) := \sup\mathbf{\mathrm{Approx}(\bar{\xi}_{n}}(H,p))$.

In order to compute an approximate value for $p_{crit}(H)$, we apply the following method, starting from $p = 0.99$. 
\begin{enumerate}
\item We simulate the random variable $\mathrm{Approx}(\bar{r}_{10000}(H,p)) \times (10000)^{-1}$.
\item If the value obtained is larger than $-0.005$, we take $p_{crit}(H) = p$.
\item Otherwise, we substitute $p$ with $p - 0.01$, and restart at Step~$1$.
\end{enumerate}

\begin{acknowledgements}
\quad The author would like to thank her PhD supervisors Nathalie Machon and Amandine Véber for helpful discussions about the model and for their comments on the manuscript. The author is also grateful to the two anonymous reviewers for their helpful and constructive comments and suggestions. 
This work was partly supported by the chaire program "Modélisation Mathématique et Biodiversité" of Veolia Environnement-Ecole Polytechnique-Museum National d’Histoire Naturelle-Fondation X.
\end{acknowledgements}

\bibliographystyle{plain}
\bibliography{biblio_pa_bg_maths}

\begin{thebibliography}{10}

\bibitem{austerlitz1997evolution}
F.~Austerlitz, B.~Jung-Muller, B.~Godelle, and P.-H. Gouyon.
\newblock Evolution of coalescence times, genetic diversity and structure
  during colonization.
\newblock {\em Theoretical {P}opulation {B}iology}, 51(2):148--164, 1997.

\bibitem{barton2013genetic}
N.H. Barton, A.M. Etheridge, J.~Kelleher, and A.~V{\'e}ber.
\newblock Genetic hitchhiking in spatially extended populations.
\newblock {\em {T}heoretical {P}opulation {B}iology}, 87:75--89, 2013.

\bibitem{baskin2014seeds}
C.C. Baskin and J.M. Baskin.
\newblock {\em Seeds: ecology, biogeography, and, evolution of dormancy and
  germination}.
\newblock {A}cademic {P}ress, 2014.

\bibitem{blath2013ancestral}
J.~Blath, A.~Gonz{\'a}lez~Casanova, N.~Kurt, and D.~Span{\`o}.
\newblock The ancestral process of long-range seed bank models.
\newblock {\em {J}ournal of {A}pplied {P}robability}, 50(3):741--759, 2013.

\bibitem{blath2016new}
J.~Blath, A.~Gonz{\'a}lez~Casanova, N.~Kurt, and M.~Wilke-Berenguer.
\newblock A new coalescent for seed-bank models.
\newblock {\em {A}nnals of {A}pplied {P}robability}, 26(2):857--891, 2016.

\bibitem{blath2020population}
J.~Blath and N.~Kurt.
\newblock Population genetic models of dormancy.
\newblock {\em arXiv preprint arXiv:2012.00810}, 2020.

\bibitem{boenkost2020haldane}
F.~Boenkost, A.~Gonz{\'a}lez~Casanova, C.~Pokalyuk, and A.~Wakolbinger.
\newblock Haldane's formula in {C}annings models: The case of moderately strong
  selection.
\newblock {\em arXiv preprint arXiv:2008.02225}, 2020.

\bibitem{boenkost2021haldane}
F.~Boenkost, A.~Gonz{\'a}lez~Casanova, C.~Pokalyuk, and A.~Wakolbinger.
\newblock Haldane’s formula in {C}annings models: the case of moderately weak
  selection.
\newblock {\em {E}lectronic {J}ournal of {P}robability}, 26:1--36, 2021.

\bibitem{borgy2015dynamics}
B.~Borgy, X.~Reboud, N.~Peyrard, R.~Sabbadin, and S.~Gaba.
\newblock Dynamics of weeds in the soil seed bank: a hidden {M}arkov model to
  estimate life history traits from standing plant time series.
\newblock {\em PloS one}, 10(10):e0139278, 2015.

\bibitem{cordero2019general}
F.~Cordero, S.~Hummel, and E.~Schertzer.
\newblock General selection models: {B}ernstein duality and minimal ancestral
  structures.
\newblock {\em arXiv preprint arXiv:1903.06731}, 2019.

\bibitem{den2017multi}
F.~den Hollander and G.~Pederzani.
\newblock Multi-colony {W}right--{F}isher with seed-bank.
\newblock {\em {I}ndagationes {M}athematicae}, 28(3):637--669, 2017.

\bibitem{durrett1984oriented}
R.~Durrett.
\newblock Oriented percolation in two dimensions.
\newblock {\em {T}he {A}nnals of {P}robability}, 12(4):999--1040, 1984.

\bibitem{durrett2016genealogies}
R.~Durrett and W.-T.~L. Fan.
\newblock Genealogies in expanding populations.
\newblock {\em {A}nnals of {A}pplied {P}robability}, 26(6):3456--3490, 2016.

\bibitem{fahrig2003effects}
L.~Fahrig.
\newblock Effects of habitat fragmentation on biodiversity.
\newblock {\em Annual review of ecology, evolution, and systematics},
  34(1):487--515, 2003.

\bibitem{fenner1995ecology}
M.~Fenner.
\newblock Ecology of seed banks.
\newblock {\em Seed development and germination}, pages 507--528, 1995.

\bibitem{forien2017central}
R.~Forien and S.~Penington.
\newblock A central limit theorem for the spatial {L}ambda-{F}leming-{V}iot
  process with selection.
\newblock {\em {E}lectronic {J}ournal of {P}robability}, 22:1--68, 2017.

\bibitem{freville2013inferring}
H.~Fr{\'e}ville, R.~Choquet, R.~Pradel, and P.-O. Cheptou.
\newblock Inferring seed bank from hidden {M}arkov models: new insights into
  metapopulation dynamics in plants.
\newblock {\em Journal of {E}cology}, 101(6):1572--1580, 2013.

\bibitem{casanova2020lambda}
A.~Gonz{\'a}lez~Casanova and C.~Smadi.
\newblock On {L}ambda-{F}leming--{V}iot processes with general
  frequency-dependent selection.
\newblock {\em {J}ournal of {A}pplied {P}robability}, 57(4):1162--1197, 2020.

\bibitem{casanova2018duality}
A.~Gonz{\'a}lez~Casanova and D.~Span{\`o}.
\newblock Duality and fixation in {X}i-{W}right--{F}isher processes with
  frequency-dependent selection.
\newblock {\em {A}nnals of {A}pplied {P}robability}, 28(1):250--284, 2018.

\bibitem{gotelli1991metapopulation}
N.J. Gotelli.
\newblock Metapopulation models: the rescue effect, the propagule rain, and the
  core-satellite hypothesis.
\newblock {\em {T}he {A}merican {N}aturalist}, 138(3):768--776, 1991.

\bibitem{greven2020spatial}
A.~Greven, F.~den Hollander, and M.~Oomen.
\newblock Spatial populations with seed-bank: well-posedness, duality and
  equilibrium.
\newblock {\em arXiv preprint arXiv:2004.14137}, 2020.

\bibitem{hallatschek2008gene}
O.~Hallatschek and D.R. Nelson.
\newblock Gene surfing in expanding populations.
\newblock {\em Theoretical {P}opulation {B}iology}, 73(1):158--170, 2008.

\bibitem{hanski1997metapopulation}
I.A. Hanski, M.E. Gilpin, and D.E. McCauley.
\newblock {\em Metapopulation biology}, volume 454.
\newblock Elsevier, 1997.

\bibitem{hartarsky2021generalised}
I.~Hartarsky and R.~Szab{\'o}.
\newblock Generalised oriented site percolation, probabilistic cellular
  automata and bootstrap percolation.
\newblock {\em arXiv preprint arXiv:2103.15621}, 2021.

\bibitem{kaj2001coalescent}
I.~Kaj, S.M. Krone, and M.~Lascoux.
\newblock Coalescent theory for seed bank models.
\newblock {\em {J}ournal of {A}pplied {P}robability}, pages 285--300, 2001.

\bibitem{kimura1964stepping}
M.~Kimura and G.H. Weiss.
\newblock The stepping stone model of population structure and the decrease of
  genetic correlation with distance.
\newblock {\em Genetics}, 49(4):561, 1964.

\bibitem{lambert2015coalescent}
A.~Lambert and C.~Ma.
\newblock The coalescent in peripatric metapopulations.
\newblock {\em {J}ournal of {A}pplied {P}robability}, 52(2):538--557, 2015.

\bibitem{lennon2020principles}
J.~T. Lennon, F.~den Hollander, M.~Wilke-Berenguer, and J.~Blath.
\newblock Principles of seed banks: complexity emerging from dormancy.
\newblock {\em arXiv preprint arXiv:2012.00072}, 2020.

\bibitem{levins1969some}
R.~Levins.
\newblock Some demographic and genetic consequences of environmental
  heterogeneity for biological control.
\newblock {\em {A}merican {E}ntomologist}, 15(3):237--240, 1969.

\bibitem{louvet2021k}
A.~Louvet.
\newblock The k-parent spatial {L}ambda-{F}leming-{V}iot process as a
  stochastic measure-valued model for an expanding population.
\newblock {\em arXiv preprint arXiv:2103.02902}, 2021.

\bibitem{louvet2021detecting}
A.~Louvet, N.~Machon, J.-B. Mihoub, and A.~Robert.
\newblock Detecting seed bank influence on plant metapopulation dynamics.
\newblock {\em {M}ethods in {E}cology and {E}volution}, 12(4):655--664, 2021.

\bibitem{macarthur1967theory}
R.H. MacArthur and E.O. Wilson.
\newblock {\em The theory of island biogeography}.
\newblock {P}rinceton {U}niversity {P}ress, 1967.

\bibitem{moilanen1999patch}
A.~Moilanen.
\newblock Patch occupancy models of metapopulation dynamics: efficient
  parameter estimation using implicit statistical inference.
\newblock {\em Ecology}, 80(3):1031--1043, 1999.

\bibitem{moilanen2004spomsim}
A.~Moilanen.
\newblock {SPOMSIM}: software for stochastic patch occupancy models of
  metapopulation dynamics.
\newblock {\em Ecological modelling}, 179(4):533--550, 2004.

\bibitem{peischl2015expansion}
S.~Peischl and L.~Excoffier.
\newblock Expansion load: recessive mutations and the role of standing genetic
  variation.
\newblock {\em Molecular {E}cology}, 24(9):2084--2094, 2015.

\bibitem{pluntz2018general}
M.~Pluntz, S.~Le~Coz, N.~Peyrard, R.~Pradel, R.~Choquet, and P.-O. Cheptou.
\newblock A general method for estimating seed dormancy and colonisation in
  annual plants from the observation of existing flora.
\newblock {\em Ecology letters}, 21(9):1311--1318, 2018.

\bibitem{skala2013hypergeometric}
M.~Skala.
\newblock Hypergeometric tail inequalities: ending the insanity.
\newblock {\em arXiv preprint arXiv:1311.5939}, 2013.

\bibitem{taylor2009coalescent}
J.~Taylor and A.~V{\'e}ber.
\newblock Coalescent processes in subdivided populations subject to recurrent
  mass extinctions.
\newblock {\em {E}lectronic {J}ournal of {P}robability}, 14:242--288, 2009.

\bibitem{wakeley2001gene}
J.~Wakeley and N.~Aliacar.
\newblock Gene genealogies in a metapopulation.
\newblock {\em Genetics}, 159(2):893--905, 2001.

\bibitem{wright1931evolution}
S.~Wright.
\newblock Evolution in {M}endelian populations.
\newblock {\em Genetics}, 16:97--159, 1931.

\bibitem{vzivkovic2012germ}
D.~{\v{Z}}ivkovi{\'c} and A.~Tellier.
\newblock Germ banks affect the inference of past demographic events.
\newblock {\em Molecular {E}cology}, 21(22):5434--5446, 2012.

\end{thebibliography}

\end{document}